\documentclass[11pt, article]{amsart}
\usepackage{amsmath}
\usepackage{amsthm}
\usepackage{amssymb}
\usepackage{amscd}
\usepackage{amsfonts}
\usepackage{amsbsy}
\usepackage{graphicx}
\usepackage{array}
\usepackage{color}
\usepackage{epsfig}
\usepackage{url}
\usepackage{overpic}
\usepackage{epstopdf}
\usepackage{setspace}
\usepackage{tikz}
\usepackage[lmargin=2.5cm,tmargin=2.5cm,bmargin=2.5cm,rmargin=2.5cm]{geometry}
\usepackage{indentfirst}
\usepackage[hidelinks]{hyperref}
\hypersetup{colorlinks=true,linkcolor=red,citecolor=green,urlcolor=blue}
\usepackage{multicol}
\usepackage{footnote}
\usepackage{subfigure}

\newcommand{\R}{\ensuremath{\mathbb{R}}}
\newcommand{\la}{\lambda}

\newcommand{\vs}{\vspace{0,5cm}}

\def\e{\varepsilon}
\newcommand{\dis}{\displaystyle}
\newcommand{\sgn}{\mbox{\normalfont sign\,}}

\newtheorem {theorem} {Theorem}
\newtheorem {definition} {Definition}

\newtheorem {lemma}{Lemma}

\newtheorem {remark} {Remark}

\begin{document}
\onehalfspacing

%Revistas para submeter:
%NONLINEARITY (BRISTOL)
%JOURNAL OF NONLINEAR SCIENCE
%JDE
%MATHEMATISCHE ZEITSCHRIFT
%PROCEEDINGS OF THE ROYAL SOCIETY OF EDINBURGH. SECTION A. MATHEMATICS

\title[Limit cycles in a generalized Rayleigh Li\'enard oscillator] {Lower bounds for the number of limit cycles in a generalized Rayleigh-Li\'enard oscillator}

\author[R. D. Euz\'ebio, J. Llibre, D. J. Tonon]{Rodrigo D. Euz\'ebio$^1$, Jaume Llibre$^2$ and Durval J. Tonon$^1$}

\address{$^1$ Institute of Mathematics and Statistics of Federal University of Goi\'{a}s, Avenida Esperan\c{c}a s/n, Campus Samambaia, 74690-900, Goi\^{a}nia, Goi\'{a}s, Brazil}

\address{$^2$ Departament de Matem\`{a}tiques, Universitat Aut\`{o}noma de Barcelona, 08193 Bellaterra, Catalunya-Barcelona, Spain.}

\email{euzebio@ufg.br} \email{jllibre@mat.uab.cat} \email{djtonon@ufg.br}



\maketitle

\begin{abstract}
In this paper a generalized Rayleigh-Li\'enard oscillator is consider and lower bounds for the number of limit cycles bifurcating from weak focus equilibria and saddle connections are provided. By assuming some open conditions on the parameters of the considered system the existence of up to twelve limit cycles is provided. More precisely, the approach consists in perform suitable changes in the sign of some specific parameters and apply Poincar\'e-Bendixson Theorem for assure the existence of limit cycles. In particular, the method for obtaining the limit cycles through the referred approach is explicitly exhibited. The main techniques applied in this study are the Lyapunov constants and the Melnikov method.
\end{abstract}

\section{Introduction}\label{Secao-Introducao}

\subsection{Historical facts and equations of Rayleigh and Li\'enard}

Ordinary differential equations (ODEs) have been largely studied in mathematics since the invention of Calculus back in 17th century. Since that the theory have proved to be very accurate to model real problems from mechanics movements and chemical reactions to social and financial sciences. The interest by ODEs gained even more attraction after the remarkable work of Poincar\'e entitled {\it M\'emoire sur les courbes d\'efinies par une \'equation diff\'erentielle}, see \cite{Poincare}. This paper, dated 1882, is consider one of the starting points of the so called qualitative theory of ODEs. In particular, Poincar\'e formally introduced the concept of limit cycle, an isolated periodic orbit inside the set of all periodic orbits of an ODE, and exhibited an {\it ad hoc} example of ODE presenting a limit cycle without any connection to some concrete problem. However, the first reported case of a limit cycle surging from a real model ODE was probably provided by Rayleigh in 1877 in his study on the oscillations of a violin string, see \cite{Rayleigh}. Posteriorly in 1908 another example of limit cycle emerged from a series of works of Poincar\'e addressing wireless telegraphy, although the most recognized example of a limit cycle is due to Van der Pol on the electrical circuits in 1927, see \cite{VanDerPol}.

The goal of this paper is to study the existence and the number of limit cycles in a system which is a generalization of the Rayleigh and Van der Pol systems. The equation proposed by Rayleigh which is nowadays known as {\it Rayleigh equation} is
\begin{equation}\label{RayleighSystem}
\ddot{x}+ax+\varepsilon(c_3+c_4\dot{x}^2)\dot{x}=0
\end{equation}
where $\varepsilon$ is a small parameter. We also point to \cite{Leissa} for some historical facts about Rayleigh work.

For the sake of applications of non-linear systems it is usually interesting to assume that the unperturbed part of \eqref{RayleighSystem} has a potential of the form $V(x)=ax^2+bx^4$ so the total energy is $\dot{x}^2+V(x)$. This is achieved by adding the term $2bx^3$ to the last equation obtaining then the equation
\begin{equation}\label{GeneralRayleighSystem}
\ddot{x}+ax+2bx^3+\varepsilon(c_3+c_4\dot{x}^2)\dot{x}=0,
\end{equation}
which is the {\it generalized Rayleigh system}. By replacing the term $c_4\dot{x}^2$ in the Rayleigh equations by $c_2x^2+c_1x^4$ we obtain the famous generalized Van der Pol or Li\'enard equation. Therefore, by combining generalized Rayleigh and Li\'enard equations we obtain
\begin{equation}\label{GeneralMixedRayleighLienardSystem}
\ddot{x}+ax+2bx^3+\varepsilon(c_3+c_2x^2+c_1x^4+c_4\dot{x}^2)\dot{x}=0.
\end{equation}
Last equation is common referred in the literature as {\it mixed generalized Rayleigh-Li\'enard equation}. In the general case, one can study the problem
\begin{equation*}\label{GeneralSystem}
\ddot{x}+ax+2bx^3=\varepsilon f(x,\dot{x}),
\end{equation*}
see for instance the work of Guckenheimer and Holmes in \cite{GH}.

In this paper we study the limit cycles for the case $f(x,\dot{x})=(c_3+c_2x^2+c_1x^4+c_4\dot{x}^2+c_5 x^6+ c_6 \dot{x}^4)\dot{x}$ so that the mixed generalized Rayleigh-Li\'enard equation becomes a particular case of the equation we deal with. The equation we study is then
\begin{equation}\label{eq-R-Lgeneralizado}
\ddot{x}+ax+2b x^3-\varepsilon(c_3 + c_2 x^2+ c_1 x^4+c_4 \dot{x}^2 + c_5 x^6+ c_6 \dot{x}^4)\dot{x}=0,
\end{equation}
that is equivalent to the system
\begin{equation}\label{sistema-R-Lgeneralizado}
%\left\{
\begin{array}{ll}
\dot{x}   &=y,\\
\dot{y}  &= -ax-2b x^3 + \varepsilon Q(x,y).
\end{array}
%\right.
\end{equation}
where $Q(x,y)=(c_3 + c_2 x^2+ c_1 x^4+c_4 y^2 + c_5 x^6+ c_6 y^4)y$.
Our purpose is to obtain a lower bound for the number of limit cycles of \eqref{sistema-R-Lgeneralizado}.

\subsection{Related literature and main goals of the paper}
The problem of finding limit cycles involves several methods and approaches, some of them are briefly summarized in what follows. In \cite{Bejarano} the authors study equation \eqref{GeneralMixedRayleighLienardSystem} by using the harmonic balance and Krylov-Bogoliubov methods to obtain up to two limit cycles. System \eqref{GeneralMixedRayleighLienardSystem} was also studied by Lynch in \cite{Lynch} by using Lyapunov constants obtaining three limit cycles. Then, the same authors of \cite{Bejarano} using harmonic balance method and elliptic functions presented an example with seven limit cycles for the case $a<0$ and $b>0$, see \cite{Bejarano2}. In \cite{Wu-Han} the authors prove that system \eqref{GeneralMixedRayleighLienardSystem} can have eight limit cycles, improving the lower bound obtained in \cite{Bejarano2}. The authors use both Lyapunov constants as well as Melnikov method to get the results for the case $ab<0$. It also worth to mention the recent work \cite{Chen} on which the authors study the global dynamics from a particular case of system \eqref{GeneralRayleighSystem}, namely, assuming $\varepsilon=c_4=2b=1$. The main result of that paper provides conditions on $a$ and $c_3$ so that system \eqref{GeneralRayleighSystem} presents pitchfork, Hopf, homoclinic and double limit cycle bifurcations, which would be a very difficult - if possible - task considering arbitrary parameters or a more general system \eqref{GeneralMixedRayleighLienardSystem}.

In this paper we study system \eqref{sistema-R-Lgeneralizado} which generalizes equation \eqref{GeneralMixedRayleighLienardSystem}. We consider the case where $a$ and $b$ have opposite sign and apply both Lyapunov and Melnikov methods as the authors did in \cite{Wu-Han}. The approach consists in to consider suitable perturbations on the coefficients of the system to produce changes in the sign of the Lyapunov constants (derivatives of the Poincar\'e map) and Melnikov function. The limit cycles essentially bifurcate from weak focus equilibria and from heteroclinic and homoclinic loops containing those weak focus.

The main contributions of this paper can be summarized as follows: first, we consider a quite general system without assuming any hypotheses on the parameters except the condition $ab<0$. So we are able to obtain from one to twelve limit cycles in a region of the phase portrait containing the equilibria. On the other hand, we explicitly exhibit the algorithm to obtain those limit cycles and we provide the conditions for the realization of those number of limit cycles. This important part of the process for obtaining the limit cycles is in general omitted in the literature. This is the case in \cite{Wu-Han}, where the authors only pointed out the approach which {\it a priori} does not guarantee the realization of the limit cycles as claimed. We stress, however, that although paper \cite{Wu-Han} contain some minor miscalculations, assuming $c_5=c_6=0$ our results points to an upper bound of eight limit cycles, which is the quantity obtained in that paper.

\subsection{The main result of the paper}
We distinguish between the limit cycles that bifurcate from the weak focus equilibria from the ones that bifurcate from saddle connections, denoting them as small amplitude limit cycles, or large amplitude limit cycles, respectively. Therefore we say that system \eqref{sistema-R-Lgeneralizado} presents a configuration $(i,j)$ of limit cycles if there exists $i$ limit cycles of small amplitude and $j$ limit cycles of large amplitude. The main result of this paper is the following.

\begin{theorem}\label{teorema-apos}
	Consider system \eqref{sistema-R-Lgeneralizado}. Then there exists suitable values of parameters realizing the following configurations of limit cycles:
	\begin{itemize}
		\item [$(a)$] $(s,m)$ if $a>0$ and $b<0$, where $m\in \{0, 1, 2, 3\}$, $s\in \{0, 1, 2, 3, 4, 5\}$ and $s+m\leq 5$.
		\item [$(b)$] $(2s,3m+k)$ if $a<0$ and $b>0$, where $m\in \{0, 1, 2\}$, $k\in\{1,2\}$, $s\in \{0, 1, 2, 3, 4, 5\}$ and $2s+3m+k\leq 12$.
	\end{itemize}	
\end{theorem}

The rest paper is organized as follows: In Section \ref{Secao-Resultados} we present some general facts about Lyapunov constants and Melnikov method. In Section \ref{Secao-FormaCanonica} the canonical forms that we consider in this paper are presented. The Lyapunov constants and the Melnikov functions related with the generalized Rayleigh-Li\'enard system are given in Sections \ref{Secao-ConstantesLyapunov} and \ref{Secao-FuncoesMelnikov}, respectively. Finally, in Section \ref{Secao-Prova} we state and prove some auxiliary results and prove Theorem \ref{teorema-apos}.

\section{Preliminary}\label{Secao-Resultados}

\subsection{Lyapunov constants}
In the following we briefly present the approach to deal with the limit cycles emerging from changing signs in the Lyapunov constants. It can be founded for instance in \cite{Bautin, Blows-Lloyd, LivroLlibre}. Consider differential systems
\begin{eqnarray}\label{sistema-geral}
\dot{x}=\la x-y+P(x,y),\qquad \dot{y}=x+\la y+Q(x,y),
\end{eqnarray}
where $P$ and $Q$ are polynomials without constant and linear terms. Then the origin of system \eqref{sistema-geral} is a {\it weak focus} if $\la=0$. The limit cycles that bifurcate from a weak focus are called {\it small amplitude limit cycles}. 

When $\la=0$ we denote a local Lyapunov function by $V$, defined in a neighborhood of the origin. Then the origin is a weak focus it is stable or unstable if $\dot{V}<0$ or $\dot{V}>0$, respectively, where $\dot{V}$ denotes the rate of changes of $V$ along the trajectories of \eqref{sistema-geral}. The expression of $V$ can be constructed, see \cite{Blows-Lloyd} and \cite{Gobber}, and it is of the form
\[
\dot{V}=\eta_2 (x^2+y^2) +\eta_4 (x^2+y^2)^2+ \dots + \eta_{2k} (x^2+y^2)^{2k}+ \ldots,
\]
where $\eta_{2k}$ is a polynomial in the coefficients of the polynomials $P$ and $Q$. We define $\eta_{2k}$ the $k$-th {\it focal value}. In the following we get that the weak focus is stable, unstable if the first non-zero focal value is negative, positive, respectively, see \cite{Blows-Lloyd}.

When the origin is a weak focus it is a center if and only if
$\eta_{2k}=0$ for all $k$. Moreover the stability of the origin is
determined by the sign of the first non--zero focal value. As
$\eta_{2k}$ is relevant only when $\eta_{2l}=0$ for $l<k$, we put
$\eta_2=\eta_4=\ldots=\eta_{2k-2}=0$ in the expression for
$\eta_{2k}$. Instead of working with the $\eta_{2k}$ following to many authors we prefer to work with
$$V_{2k+1}=2 \pi\eta_{2k},$$
and call it the $n$--th {\it Lyapunov constant}. 

For $\e>0$ and small consider the interval $J=\{(x,0): 0\le x < \e\}$ and the {\it Poincar\'e return map} $x\mapsto h(x)$ defined from $J \to \{(x,0): 0\le x \}$. It assigns to $x$ the abscissa $h(x)$ of the point where the orbit of the differential system \eqref{sistema-geral} starting at the point $(x,0)\in J$ first returns to the positive half-axis $\{(x,0): 0\le x \}$. Then the {\it displacement function} is defined as $x\mapsto d(x)$ from $J\to\R$ by $d(x)=h(x)-x$. Therefore the orbit of system \eqref{sistema-geral} through the point $(x,0)\not= (0,0)$ is periodic if and only if $x$ is a zero of the displacement function.

Clearly the Lyapunov constants are related with the coefficients of the displacement function, because the origin of system \eqref{sistema-geral} is a center if and only if the displacement function is identically zero, if and only if the Lyapunov constants $V_{2k+1}=0$ for $k\ge 1$. In fact it is known that 
$$
d(x)= \la x \left[a_0+ \sum_{j=1}^{\infty} a_j^1 x^j \right] +\sum_{k\ge 1} V_{2k+1} x^{2k+1}\left[ 1+\sum_{j=1}^{\infty} a_j^{2k+1} x^k\right],
$$
for $|x|< \e$, where 
\[
a_0=\frac{e^{2\pi \la}-1}{\la}= 2\pi + O(\la),
\]
and the $a_j^{2k+1}$ for $k=1,2,\ldots$ are analytic functions in $\la$ and in the coefficients of the polynomials $P$ and $Q$. For more details in the displacement function see \cite{Rou}.

Since $P$ and $Q$ are polynomials, by the Hilbert basis theorem there
is a constant $m$ such that $V_{2k+1}=0$ for all $k\ge 1$ if and only if
$V_{2k+1}=0$ if $k=1,\ldots,m$. Therefore it is necessary to compute only a
finite number of the Lyapunov constants, though with few
exceptions for any given case it is unknown a priori how many are
required.  

\begin{definition}
We say that the origin is a weak focus of order $k$ of system \eqref{sistema-geral} if $\eta_2= \dots= \eta_{2k} =0$ and $\eta_{2k+2} \neq 0$.
\end{definition}

\begin{remark}
By the previous definition, the origin is a weak focus of order $k$ of system \eqref{sistema-geral} if $V_3= \dots=V_{2k-1} =0$ and $V_{2k+1} \neq 0$.
\end{remark}
In \cite{Blows-Lloyd} the authors stated that if a system presents a weak focus at the origin of order $k$ at most $k$ small amplitude limit cycles can bifurcate from the origin under perturbation of the system. In the following we describe briefly how to provide a convenient perturbation of the original system presenting a weak focus of order $k$ at the origin to produce $k$ small amplitude limit cycles.

The expressions $\eta_{2k}$ provides the Lyapunov constants $V_{2k-1}$. We assume that $V_1=1$, $V_3= \dots = V_{2k-1} =0$ and $V_{2k+1} \neq 0$. Without loss of generality we assume that $V_{2k+1} <0$. Therefore the origin is stable. Consider $\Gamma_1$ be a level curve of $V$ which is sufficiently close to the origin. So the flow of system \eqref{sistema-geral} is inward across it.

Now consider a suitable perturbation $\mathit{S}_1$ of system \eqref{sistema-geral} such that the Lyapunov constants satisfy $V_3= \dots = V_{2k-3} =0$ and $V_{2k-1}>0$. Therefore, the origin now is unstable for $\mathit{S}_1$. As $\mathit{S}_1$ is sufficiently close to system \eqref{sistema-geral} then the flow remains inward across $\Gamma_1$. Consider $W_1$ the Lyapunov function of $\mathit{S}_1$, take $\Gamma_2$ a level curve of $W_1$, inside of the region limited by $\Gamma_1$ and sufficiently close to the origin in such a way that the flow of $\mathit{S}_1$ is outward across of $\Gamma_2$. Therefore by the Poincar\'e-Bendixson Theorem we conclude that there exists a limit cycle of $\mathit{S}_1$ between $\Gamma_1$ and $\Gamma_2$. In the following working in a similar way we consider a convenient perturbation $\mathit{S}_2$ of $\mathit{S}_1$ with analogous properties.

In this way at most $k$ limit cycles can be produced by convenient perturbations of system \eqref{sistema-geral}. Note that the Lyapunov constants must satisfy:
\[
\begin{array}{l}
|V_{2i-1}| \ll |V_{2i+1}|\quad\mbox{and}\quad
V_{2i-1}  \cdot V_{2i+1}<0,
\end{array}
\]
for $i=3, \dots, k$. If all the constants $V_{2k+1}$ are zero then the origin is a center.

\subsection{Melnikov method}
In order to present the Melnikov method, note that system \eqref{sistema-R-Lgeneralizado} with $\varepsilon=0$ is a Hamiltonian system with the Hamiltonian function
\[
H(x,y)=\dfrac{1}{2}(y^2 + ax^2 + bx^4).
\]
If $a b>0$ then the origin is the unique equilibrium point and if $ab<0$ there exists three equilibrium points: $O=(0,0), p_1 = \left(  \sqrt{\frac{-a}{2b}}, 0\right)$ and $p_2=\left(-  \sqrt{\frac{-a}{2b}} , 0\right)$.
In this way if $a>0$ and $b<0$, then $O=(0,0)$ is a center point and $p_1,p_2$ are saddle points, and if $a<0$ and $b>0$ then $O=(0,0)$ is saddle point and $p_1,p_2$ are center points. As system \eqref{sistema-R-Lgeneralizado} with $\varepsilon=0$ is Hamiltonian, then for the case $a>0$ and $b<0$ we have a heteroclinic loop between the two saddle points $p_1$ and $p_2$. On the other hand in the case $a<0$ and $b>0$ the system presents a homoclinic loop. It is important to note that system \eqref{sistema-R-Lgeneralizado} is invariant under the transformation $(x,y)\mapsto (-x,-y)$, so its phase portrait is symmetric with respect to the origin, see Figure \ref{CampoHamiltoniano}.
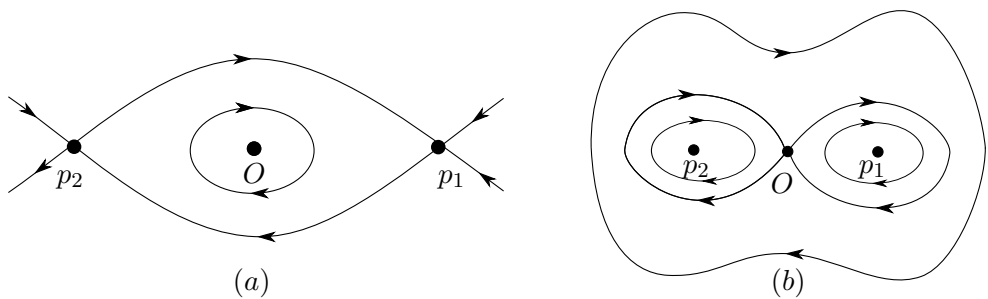
\begin{figure*}[h]
	\begin{center}
		\begin{overpic}[width=13cm]{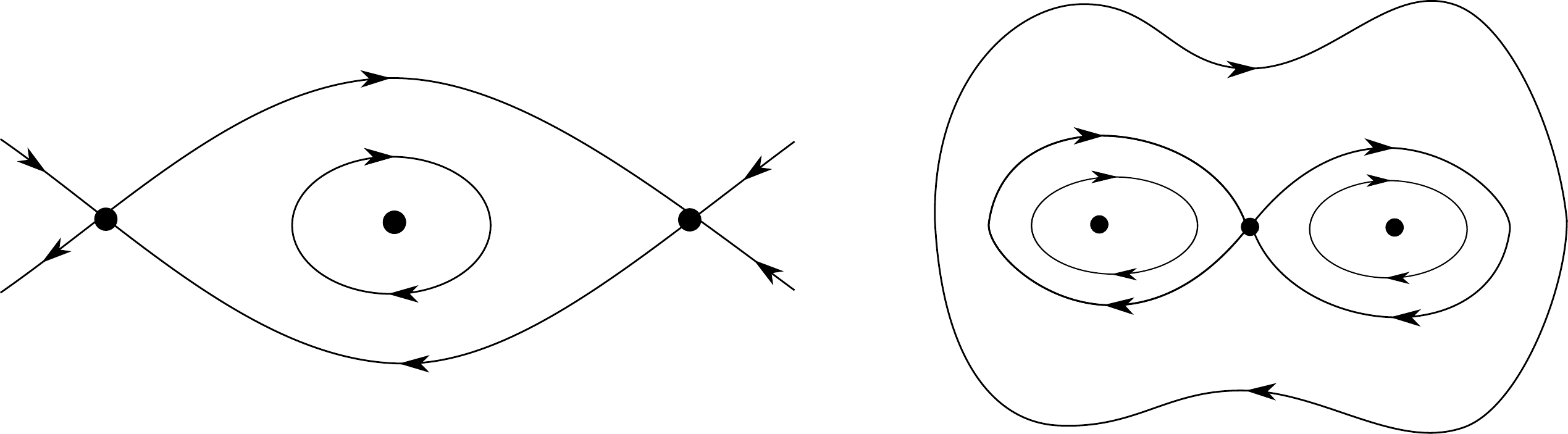}
			%\begin{overpic}[grid,tics=10,width=13cm]{CampoHamiltoniano.pdf}
			\put(5,10){$p_2$}\put(24,10){$O$}\put(44,10){$p_1$}
			\put(69,11){$p_2$}\put(78,9){$O$}\put(87,11){$p_1$}
			\put(23,-1){$(a)$}\put(78,-1){$(b)$}
		\end{overpic}
	\end{center}
	\caption{The topological structure of system \eqref{sistema-R-Lgeneralizado} with $\varepsilon=0$. In $(a)$ we get $a>0, b<0$, and in $(b)$ we have $a<0, b>0$.}\label{CampoHamiltoniano}
\end{figure*}

We note that for $\varepsilon\neq0$ system \eqref{sistema-R-Lgeneralizado} is no longer Hamiltonian so if $\varepsilon$ is sufficiently small generically the saddle connections are broken and the center structure is destroyed. Therefore in both scenarios we can eventually have the birth of limit cycles emerging from those equilibriums or loops. In this direction we devote two sections to study both center-focus problem and persistence of homoclinic and heteroclinic loops.

\section{Canonical forms and topological structure for the generalized mixed Rayleigh-Li\'enard oscillator}\label{Secao-FormaCanonica}

\subsection{Canonical forms for the generalized mixed Rayleigh-Li\'enard oscillator}

In the following we apply a linear change of coordinates to equation \eqref{eq-R-Lgeneralizado}, providing a simpler expression for the system that models a generalized mixed Rayleigh-Li\'enard oscillator.

\begin{lemma}
	Consider system \eqref{sistema-R-Lgeneralizado}. The following statements holds.
	\begin{itemize}
		\item [$(a)$] If $a>0$ and $b<0$ then system \eqref{sistema-R-Lgeneralizado} is topologically equivalent to system
		\begin{equation}\label{FormaCanonicaR-L-apos}
	    \begin{array}{ll}
		\dot{x}  &=y,\vspace{0.2cm}\\
		\dot{y}  &= -x- \dfrac{2 b}{a^2} x^3 + \varepsilon Q_1(x,y),
		\end{array}
		\end{equation}
		where $Q_1(x,y)=(d_3 + d_2 x^2+ d_1 x^4 + d_4 y^2 + d_5 x^6+ d_6 y^4)y$, $d_1=c_1/\sqrt[5]{a},\, d_2=c_2/\sqrt[3]{a},/, d_3=c_3/\sqrt{a},\, d_4=c_4/\sqrt{a},\, d_5=c_5/\sqrt[7]{a}$ and $d_6=c_6/\sqrt{a}$.\\
		
		\item [$(b)$] If $a<0$ and $b>0$ then system \eqref{sistema-R-Lgeneralizado} is topologically equivalent to system
		\begin{equation}\label{FormaCanonicaR-L-aneg}
		\begin{array}{ll}
		\dot{x}   &=y,\vspace{0.2cm}\\
		\dot{y}  &= -x + \frac{3 \sqrt{b}}{2a} x^2 -\frac{b}{2a^2} x^3  + \varepsilon Q_2(x,y),
		\end{array}
\end{equation}
where $Q_2(x,y)=(e_3 + e_7 x + e_2 x^2+ e_8 x^3 +  e_1 x^4 + e_9 x^5 + e_5 x^6 +e_4 y^2  + e_6 y^4)y$ and
		$$
		\begin{array}{l}
		e_1 = \dfrac{2 c_1 b-15 a c_5}{8 \sqrt{2} (-a)^{5/2} b}, \;\, e_2=\dfrac{15 a^2 c_5-12 a c_1 b+4 c_2 b^2}{8 \sqrt{2} (-a)^{3/2} b^2},  \;\, e_3=-\dfrac{a^3 c_5-2 a^2 c_1 b+4 a c_2 b^2-8 c_3 b^3}{8 \sqrt{2} \sqrt{-a} b^3}\\\\
		e_4=\dfrac{c_4}{\sqrt{-2a}},\;\, e_5=\dfrac{c_5}{8 \sqrt{2} (-a)^{7/2}},\;\, e_6=\dfrac{c_6}{\sqrt{-2a}},\;\, e_7 = \dfrac{3 a^2 c_5-4 a c_1 b+4 c_2 b^2}{4 \sqrt{2} \sqrt{-a} b^{5/2}}, e_8=-\dfrac{ (5 a c_5-2 c_1 b)}{2 \sqrt{2} (-a)^{3/2} b^{3/2}}, \\\\
		e_9 =\dfrac{3 c_5 }{4 \sqrt{2} (-a)^{5/2} \sqrt{b}}		.
		\end{array}
		$$
	\end{itemize}
\end{lemma}

\begin{proof}
Consider system \eqref{sistema-R-Lgeneralizado}. We initially consider the case $a>0$ and $b<0$. Applying the change of coordinates $\widetilde{x} = \sqrt{a} x, \widetilde{y} = y$ and $\widetilde{t} = \sqrt{a} t$ we obtain system \eqref{FormaCanonicaR-L-apos}, where we have removed the tilde in the expression of the system.
	
For the case $a<0$ and $b>0$, we initially translate the point $p_1=\left(  \sqrt{\frac{-a}{2b}}, 0\right)$ to the origin by the change of coordinates $\widetilde{x} = x- \sqrt{\frac{-a}{2b}}$ and $\widetilde{y} = y$. So system \eqref{sistema-R-Lgeneralizado} becomes
\begin{equation}\label{sistema-transladado}
\begin{array}{ll}
\dot{x}   &=y,\vspace{0.2cm}\\
\dot{y}  &= x \left(3 \sqrt{-2 a b} x + 2 a - 2 b x^2\right) +  \varepsilon Q(x,y),
\end{array}
\end{equation}
where $Q(x,y)= \frac{y }{8 b^3}  \Big(8  \sqrt{-2ab} ab  x \left(c_1+5 c_5 x^2\right)$ $-6 \sqrt{-2ab} a^{2} c_5  x - a^3 c_5 + 2 a^2 b \left(c_1+15 c_5 x^2\right) $\linebreak $- 8  \sqrt{-2ab}  b^{2} x \left(2 c_1 x^2 + c_2 + 3 c_5 x^4\right) - 4 a b^2 (6 c_1 x^2 + c_2$ $ + 15 c_5 x^4) + 8 b^3 (c_1 x^4 + c_2 x^2 + c_3 + c_4 y^2 + c_5 x^6$ $+c_6 y^4)\Big)$ and as in the previous case, we remove the tildes in the expression of the system.

After we consider the rescaling $\widetilde{\widetilde{x}} = \sqrt{-2 a} \widetilde{x}, \widetilde{\widetilde{y}} = \widetilde{y}$ and $\widetilde{t} = \sqrt{-2a} t$. Applying this change of coordinates to system \eqref{sistema-transladado}, we get system \eqref{FormaCanonicaR-L-aneg}, where we remove the tildes in the expression of the system. This completes the proof of the lemma.
\end{proof}

\subsection{Topological structure of the canonical forms}\label{topological_structure}

We first note that both systems \eqref{FormaCanonicaR-L-apos} and $\eqref{FormaCanonicaR-L-aneg}$ still have three equilibrium points and they are also Hamiltonian when $\varepsilon=0$. System \eqref{FormaCanonicaR-L-apos} has the same equilibrium points than system \eqref{sistema-R-Lgeneralizado} but in this case the Hamiltonian function associated to its unperturbed part is given by
$$
H_1(x,y)=ax^2+\frac{b}{2}x^4+\frac{1}{2}y^2.
$$
Since the orbits of system \eqref{FormaCanonicaR-L-apos} with $\varepsilon=0$ lie on the $h$-levels of function $H_1$, we see that for $h=-a^2/8b$ it has a heteroclinic loop $\mathit{A}$ formed by the saddle equilibrium points and two orbits $\mathit{A}_1$ and $\mathit{A}_2$ connecting them, that is, $\mathit{A}=\{p_1\}\cup\mathit{A}_1\cup\{p_2\}\cup\mathit{A}_2$. From the expression of $H_1$ we obtain the analytical expression of these arcs:
$$
\mathit{A}_{1,2}: y=y(x)=\mp\sqrt{\dfrac{-b}{a}}x^2\pm\sqrt{\dfrac{a}{-4b}},
$$
where $|x|\leq \sqrt{-\frac{a}{2b}}$. In this case we note that $\mathit{A}$ surrounds the center equilibrium $O=(0,0)$.

On the other hand the points $O=(0,0)$, $q_1=\left(\frac{a}{\sqrt{b}},0\right)$ and $q_2=\left(\frac{2a}{\sqrt{b}},0\right)$ are equilibrium points for system $\eqref{FormaCanonicaR-L-aneg}$ being
$$
H_2(x,y)=\frac{x^2}{2}-\frac{\sqrt{b}}{2a}x^3+\frac{b}{8a^2}x^4+\frac{1}{2}y^2
$$
the Hamiltonian function of the system for $\varepsilon=0$. Doing a translation of the original system for obtaining system $\eqref{FormaCanonicaR-L-aneg}$, $q_1$ is a saddle when $\varepsilon=0$, the equilibria $O$ and $q_2$ are centers. The level $h=a^2/8b$ of Hamiltonian function $H_2$ contains two homoclinic loops $\mathit{L}_l$ and $\mathit{L}_r$ and the saddle point centered at $q_1$. We call $\mathit{L}=\mathit{L}_l\cup\{q_1\}\cup\mathit{L}_r$. In this case $\mathit{L}_l$ and $\mathit{L}_r$ are determined by the following relations:
$$
\mathit{L}_l: y=\sgn(y)(-a+\sqrt{b}x)\sqrt{\dfrac{a^2+2a\sqrt{b}x-bx^2}{a^2b}},
$$
if $\frac{a(1+\sqrt{2})}{\sqrt{b}}\leq x \leq \frac{a}{\sqrt{b}}$ and
$$
\mathit{L}_r:y=\sgn(y)(a-\sqrt{b}x)\sqrt{\dfrac{a^2+2a\sqrt{b}x-bx^2}{a^2b}},
$$
if $\frac{a}{\sqrt{b}}\leq x \leq \frac{a(1-\sqrt{2})}{\sqrt{b}}$. Observe that the homoclinic loop $\mathit{L}_i$ surrounds the equilibria $O$ and $q_2$.

\section{Lyapunov constants}\label{Secao-ConstantesLyapunov}

In this section we present the Lyapunov constants for each system.

\begin{lemma}\label{lema-coeficientesLyapunov-apos}
Consider system \eqref{FormaCanonicaR-L-apos} with $d_3=0$. The first Lyapunov constant is $V_1=1$, then
%\begin{equation}\label{V1-apos}
%V_1=e^{ \dfrac{2 \pi d_3 \varepsilon }{\sqrt{4 - d_3^2 \varepsilon ^2}}}.  
%\end{equation}
	\begin{equation}\label{V3-apos}
	V_3 = -\dfrac{1}{4} \pi  \varepsilon  (d_2+3 d_4).
	\end{equation}
	If $V_3=0$, then
	\begin{equation}\label{V5-apos}
	V_5= -\dfrac{\pi  \varepsilon  (a^2 d_1+5 a^2 d_6+6b d_4)}{8 a^2}.
\end{equation}
If $V_3=V_5=0$, then
\begin{equation}\label{V7-apos}
V_7= -\dfrac{\pi  \varepsilon  \left(-6 a^2 d_4^3 \varepsilon^2+5 a^2 d_5+80 b d_6\right)}{64 a^2}.
	\end{equation}
	If $V_3=V_5=V_7=0$, then
	\begin{equation}\label{V9-apos}
V_9= \dfrac{3 \pi  \varepsilon  \left(15 a^4 d_4^2 d_6 \epsilon ^2+12 a^2 b d_4^3 \varepsilon ^2 - 35 b^2 d_6\right)}{160 a^4}	
\end{equation}
	If $V_3=V_5=V_7=V_9=0$, then
	\begin{equation}\label{V11-apos}
V_{11}= -\dfrac{3 \pi  d_4^3 \varepsilon ^3}{3200 a^4 \left(3 a^4 d_4^2 \varepsilon ^2-7 b^2\right)^2} \left(315 a^{12} d_4^6 \varepsilon ^6 + 2644 a^8 b^2 d_4^4 \varepsilon ^4 - 9065 a^4 b^4 d_4^2 \varepsilon ^2 - 7350 b^6\right).	
 \end{equation}	
\end{lemma}
\begin{proof}
	Consider the polynomial system
	\[
	\begin{array}{ll}
	\dot{x}  &=y+p(x,y),\\
	\dot{y}  &=-x+q(x,y).
	\end{array}
	\]
	Applying the change of coordinates  $x=r \cos\theta$ and $y=\sin\theta$, we can write
	\begin{equation}\label{eqpolar-sistema}
	\dfrac{dr}{d\theta} =\dis \sum_{i\geq 1} v_i(\theta) r^i,
	\end{equation}
	where $v_i(\theta)$ are trigonometric polynomials in the variables $\cos\theta$ and $\sin\theta$. Denoting by $r(\theta, r_0)$ the solution of \eqref{eqpolar-sistema} satisfying $r(0, r_0)=r_0$, then in a neighborhood of $r=0$ we obtain
	\[
	r(\theta, r_0)= u_1(\theta) r_0 + \dis \sum_{i\geq 2} u_i(\theta) r_0^i,
	\]
	with $u_i(0)=0$ for all $i\geq 2$, then the Poincar\'e return map is
	\[
	\Pi(r_0)=r(2\pi, r_0) =u_1(2\pi) r_0 + \dis \sum_{i\geq 2} u_i(2\pi ) r_0^i.
	\]
	As the Poincar\'e map is analytic the condition $\Pi(r_0)\equiv r_0$ is equivalent to the fact that system \eqref{FormaCanonicaR-L-apos} presents a center at the origin. Note that $\Pi(r_0)\equiv r_0$ if and only if $u_i(2\pi )=0$ for all $i \geq 2$. As stated in \cite{Bautin} and by direct computations we obtain that $u_{2k}(2\pi)=0$ for all $k$ a positive integer.
	
	At this moment we are able to obtain the Lyapunov constants for system \eqref{FormaCanonicaR-L-apos}. Consider system \eqref{FormaCanonicaR-L-apos} and assuming that $d_3=0$, then following the steps described previously, we get that the first Poincar\'e constant
	\[
	u_1(2 \pi)=1.
	\]
%	e^{ \dfrac{2 \pi d_3 \varepsilon }{\sqrt{4 - d_3^2 \varepsilon ^2}}}
	and then we obtain
	\[
	u_3(2\pi) = -\dfrac{1}{4} \pi  \varepsilon  (d_2+3 d_4).
	\]
	Assuming that $d_2=-3 d_4$ we obtain
	\[
	u_5(2\pi)=  -\dfrac{\pi  \varepsilon  (a^2 d_1+5 a^2 d_6+6b d_4)}{8 a^2}.
	\]
	If $d_2=-3 d_4$ and $d_1 = -\frac{6 b d_4}{a^2}-5 d_6$, then
	\[
	u_7(2\pi)= -\dfrac{\pi  \varepsilon  \left(-6 a^2 d_4^3 \varepsilon^2+5 a^2 d_5+80 b d_6\right)}{64 a^2}.
	\]
	If $d_2=-3 d_4, d_1 = -\frac{6 b d_4}{a^2}-5 d_6$ and $d_5=\frac{6 d_4^3 \varepsilon ^2}{5}-\frac{16 b d_6}{a^2}$, then
	\[
u_9(2\pi)=  \dfrac{3 \pi  \varepsilon  \left(15 a^4 d_4^2 d_6 \epsilon ^2+12 a^2 b d_4^3 \varepsilon ^2 - 35 b^2 d_6\right)}{160 a^4}.
	\]
	If $d_2=-3 d_4, d_1 = -\frac{6 b d_4}{a^2}-5 d_6, d_5=\frac{6 d_4^3 \varepsilon ^2}{5}-\frac{16 b d_6}{a^2}$ and $d_6=\frac{12 a^2 b d_4^3 \varepsilon ^2}{5 \left(7 b^2-3 a^4 d_4^2 \varepsilon ^2\right)}$ then
	\[
u_{11}(2\pi)= -\dfrac{3 \pi  d_4^3 \varepsilon ^3}{3200 a^4 \left(3 a^4 d_4^2 \varepsilon ^2-7 b^2\right)^2} \left(315 a^{12} d_4^6 \varepsilon ^6 + 2644 a^8 b^2 d_4^4 \varepsilon ^4 - 9065 a^4 b^4 d_4^2 \varepsilon ^2 - 7350 b^6\right).	
\]
	As $\varepsilon$ is a small parameter then $u_{11}(2\pi)=0$ if and only if $d_4=0$. So if besides $d_3=0$ we assume $d_2=-3 d_4,\, d_1 = -\frac{6 b d_4}{a^2}-5 d_6,\, d_5=\frac{6 d_4^3 \varepsilon ^2}{5}-\frac{16 b d_6}{a^2},\, d_6=\frac{12 a^2 b d_4^3 \varepsilon ^2}{5 \left(7 b^2-3 a^4 d_4^2 \varepsilon ^2\right)}$ and $d_4=0$ then we obtain that all the others coefficients $d_i$ are null and therefore system \eqref{FormaCanonicaR-L-apos} is a Hamiltonian system, and therefore the origin is a center. In what follows we denote the Lyapunov constant $u_{i}(2\pi)$ by $V_i$.
\end{proof}

\begin{lemma}\label{lema-coeficientesLyapunov-aneg}
	Consider system \eqref{FormaCanonicaR-L-aneg} with $e_3=0$. The first Lyapunov constant is $V_1=1$, then
%{\small
%	\begin{equation}\label{V1-aneg}
%	V_1=e^{\dfrac{2 \pi  e_3 \varepsilon }{\sqrt{4-e_3^2 \varepsilon ^2}}} .
%	\end{equation}
%}
%If $V_1=1$, then
{\small
	\begin{equation}\label{V3-aneg}
	V_3= -\frac{\pi  \varepsilon  \left(2 a (e_2+3e_4)+3 \sqrt{b}e_7\right)}{8 a}.
	\end{equation}
}
	If $V_3=0$, then
{\small
\begin{equation}\label{V5-aneg}
V_5 =\dfrac{\pi  \varepsilon }{96 a^3} \Big(-4 a^3 \left(3e_1-2e_4e_7^2 \varepsilon ^2+15e_6\right)-30 a^2 \sqrt{b}e_8+117 a be_4+15 b^{3/2}e_7\Big).
\end{equation}
}
If $V_3=V_5=0$, then
{\small
\begin{equation}\label{V7-aneg}
\begin{array}{ll}
V_7   &=\dfrac{\pi  \varepsilon }{384 a^5} \Big(a^5 \left(4 \varepsilon ^2 \left(9e_4^3+12e_4e_7e_8-5e_6e_7^2\right)-30e_5\right)-21 a^4 \sqrt{b} \left(12e_4^2e_7 \varepsilon ^2+5e_9\right)\\\\
&+a^3 b \left(1005e_6-158e_4e_7^2 \varepsilon ^2\right)+210 a^2 b^{3/2}e_8-945 a b^2e_4-105 b^{5/2}e_7\Big).
\end{array}
\end{equation}
}
If $V_3=V_5=V_7=0$, then
{\small
\begin{equation}\label{V9-aneg}
\begin{array}{ll}
V_9&=\dfrac{\pi  \varepsilon }{3225600 a^7 \sqrt{b}} \Big(64 a^6e_7 \varepsilon ^4 \Big(212 a^2e_4 \Big(9e_4^3+12e_4e_7e_8-5e_6e_7^2\Big)+63 a \sqrt{b}e_7 \Big(10e_6e_7^2\\\\
&-147e_4^3\Big)+2441 be_4^2e_7^2\Big)+4725 b^{3/2} \Big(-200 a^5e_5-1665 a^3 be_6-154 a^2 b^{3/2}e_8\\\\
&+693 a b^2e_4+77 b^{5/2}e_7\Big)-30 a^3 \varepsilon ^2 \Big(3392 a^5e_4e_5e_7-336 a^4 \sqrt{b} \Big(90e_4^2e_6+11e_4e_8^2\\\\
&+16e_6e_7e_8\Big)+16 a^3 be_4 (2583e_4e_8+4910e_6e_7)-28 a^2 b^{3/2} \Big(2889e_4^3-2122e_4e_7e_8\\\\
&+699e_6e_7^2\Big)-304164 a b^2e_4^2e_7-63749 b^{5/2}e_4e_7^2\Big)\Big).
\end{array}
\end{equation}
}
If $V_3=V_5=V_7=V_9=0$, then
{\small
\begin{equation}\label{V11-aneg}
\begin{array}{ll}
V_{11}      &=K_1  \Big(-167157760 a^{12}e_4^4e_7^5 \varepsilon^8 +49116375 b^{9/2} \Big(785 a^3e_6+26 a^2 \sqrt{b}e_8-117 a be_4\\\\
&-13 b^{3/2}e_7\Big)-1536 a^9e_7 \varepsilon ^6 \Big(30 a^3 \Big(3339e_4^6-5322e_4^4e_7e_8-30950e_4^3e_6e_7^2-2368e_4e_6e_7^3e_8\\\\
&+1150e_6^2e_7^4\Big)+a^2 \sqrt{b}e_4^2e_7 \Big(-497277e_4^3-1966926e_4e_7e_8+1228180e_6e_7^2\Big)
\end{array}
\end{equation}
}

{\small
\begin{equation*}
\begin{array}{ll}
&+3 a be_4e_7^2 \Big(3023193e_4^3-46000e_6e_7^2\Big)+1441191 b^{3/2}e_4^3e_7^3\Big)-9450 a^3 b^{3/2} \varepsilon ^2 \Big(-400 a^6e_6\\\\
& \Big(3480e_4e_6+259e_8^2\Big)+200 a^5 \sqrt{b}e_6 (8724e_4e_8+3595e_6e_7)+4 a^4 b \Big(-260505e_4^2e_6\\\\
&+190248e_4e_8^2+313438e_6e_7e_8\Big)-12 a^3 b^{3/2}e_4 (689403e_4e_8+991486e_6e_7)\\\\
&+a^2 b^2 \Big(20703519e_4^3-4172844e_4e_7e_8+821669e_6e_7^2\Big)+20955447 a b^{5/2}e_4^2e_7\\\\
&+2440179 b^3e_4e_7^2\Big) -240 a^6 \varepsilon ^4 \Big(64 a^6 \Big(300e_6e_8\Big(135e_4^3+46e_6e_7^2\Big)-14595e_4^2e_6^2e_7\\\\
&+4950e_4^2e_8^3 +9824e_4e_6e_7e_8^2\Big)-48 a^5 \sqrt{b}e_4 \Big(90e_6 \Big(3465e_4^3+1208e_6e_7^2\Big)\\\\
&+111915e_4^2e_8^2+132626e_4e_6e_7e_8\Big)+8 a^4 b\Big(3018870e_4^4e_8-3191031e_4^3e_6e_7\\\\
&-1556520e_4^2e_7e_8^2-2282296e_4e_6e_7^2e_8 + 833175e_6^2e_7^3\Big)-24 a^3 b^{3/2} \Big(640710e_4^5\\\\
&-5077425e_4^3e_7e_8 -1778726e_4^2e_6e_7^2+220500e_6e_7^3e_8\Big)+2 a^2 b^2e_4e_7 \Big(-126500049e_4^3\\\\
&+27457254e_4e_7e_8+6865330e_6e_7^2\Big)+9 a b^{5/2}e_7^2 \Big(1523760e_6e_7^2-23884957e_4^3\Big)\\\\
&+11176158 b^3e_4^2e_7^3\Big)\Big),
\end{array}
\end{equation*}
}
where $K_1=\dfrac{\pi  \varepsilon}{5529600 a^9 	\Big(848 a^3e_4e_7 \varepsilon ^2+7875 b^{3/2}\Big)}$.
\end{lemma}

\begin{proof}
Considering now system \eqref{FormaCanonicaR-L-aneg} and $e_3=0$. Proceeding in a similar way as in the proof of Lemma \ref{lema-coeficientesLyapunov-apos}, the first Lyapunov constant is
\[
u_1(2 \pi) =1.
\]
%e^{ \dfrac{2 \pi e_3 \varepsilon }{\sqrt{4 - e_3^2 \varepsilon ^2}}}
and we conclude that the expression of $u_3(2\pi)$ is given in \eqref{V3-aneg}. Assuming that $e_2=-\frac{3 \left(2 ae_4+\sqrt{b}e_7\right)}{2 a}$
we conclude that the expression of $u_5(2\pi)$ is given in \eqref{V5-aneg}. If $e_2=-\frac{3 \left(2 ae_4+\sqrt{b}e_7\right)}{2 a}$ and
\[
e_1=\frac{8 a^3e_4e_7^2 \varepsilon ^2-60 a^3e_6-30 a^2 \sqrt{b}e_8+117 a be_4+15 b^{3/2}e_7}{12 a^3},
\]
we conclude that the expression of $u_7(2\pi)$ is given in \eqref{V7-aneg}. If $e_2=-\frac{3 \left(2 ae_4+\sqrt{b}e_7\right)}{2 a}$,
\[
e_1=\frac{8 a^3e_4e_7^2 \varepsilon ^2-60 a^3e_6-30 a^2 \sqrt{b}e_8+117 a be_4+15 b^{3/2}e_7}{12 a^3},
\]
and
\[
\begin{array}{ll}
e_9   &=\frac{  1} {105 a^4 \sqrt{b}} \Big( 36 a^5e_4^3 \varepsilon ^2+48 a^5e_4e_7e_8 \varepsilon ^2-30 a^5e_5-20 a^5e_6e_7^2 \varepsilon ^2 -252 a^4 \sqrt{b}e_4^2e_7 \varepsilon ^2\\\\
&-158 a^3 be_4e_7^2 \varepsilon ^2+1005 a^3 be_6+210 a^2 b^{3/2}e_8-945 a b^2e_4-105 b^{5/2}e_7\Big),
\end{array}
\]
we conclude that the expression of $u_9(2\pi)$ is given in \eqref{V9-aneg}. If $e_2=-\frac{3 \left(2 ae_4+\sqrt{b}e_7\right)}{2 a}$,
\[
\begin{array}{ll}
e_1  &= \dfrac{ 1 }{12 a^3} \Big( 8 a^3e_4e_7^2 \varepsilon ^2-60 a^3e_6-30 a^2 \sqrt{b}e_8+117 a be_4+15 b^{3/2}e_7\Big),\\\\
e_9   &=\frac{  1} {105 a^4 \sqrt{b}} \Big( 36 a^5e_4^3 \varepsilon ^2+48 a^5e_4e_7e_8 \varepsilon ^2-30 a^5e_5-20 a^5e_6e_7^2 \varepsilon ^2 -252 a^4 \sqrt{b}e_4^2e_7 \varepsilon ^2\\\\
&-158 a^3 be_4e_7^2 \varepsilon ^2+1005 a^3 be_6+210 a^2 b^{3/2}e_8-945 a b^2e_4-105 b^{5/2}e_7\Big),
\end{array}
\]
and
\[
\begin{array}{ll}
e_5  &=\frac{1}{120 a^5 \left(848 a^3e_4e_7 \varepsilon ^2+7875 b^{3/2}\right)}
\Big( 122112 a^8e_4^4e_7 \varepsilon ^4+162816 a^8e_4^2e_7^2e_8 \varepsilon ^4-67840 a^8e_4e_6e_7^3 \varepsilon ^4\\\\ &-592704 a^7 \sqrt{b}e_4^3e_7^2 \varepsilon ^4+907200 a^7 \sqrt{b}e_4^2e_6 \varepsilon ^2+110880 a^7 \sqrt{b}e_4e_8^2 \varepsilon ^2+40320 a^7 \sqrt{b}e_6e_7^4 \varepsilon ^4\\\\ &+161280 a^7 \sqrt{b}e_6e_7e_8 \varepsilon ^2+156224 a^6 be_4^2e_7^3 \varepsilon ^4-1239840 a^6 be_4^2e_8 \varepsilon ^2-2356800 a^6 be_4e_6e_7 \varepsilon ^2\\\\ &+2426760 a^5 b^{3/2}e_4^3 \varepsilon ^2-1782480 a^5 b^{3/2}e_4e_7e_8 \varepsilon ^2+587160 a^5 b^{3/2}e_6e_7^2 \varepsilon ^2+9124920 a^4 b^2e_4^2e_7 \varepsilon ^2\\\\ &+1912470 a^3 b^{5/2}e_4e_7^2 \varepsilon ^2-7867125 a^3 b^{5/2}e_6-727650 a^2 b^3e_8+3274425 a b^{7/2}e_4+363825 b^4e_7\Big)
\end{array}
\]
then $u_{11}(2\pi)$ is given by the expression of $V_{11}$ given in \eqref{V11-aneg}.

As $\varepsilon$ is a small parameter then the signal of $V_{11}$ is given by the sign of the expression
\begin{equation}\label{sinal-V11-CasoAneg}
\begin{array}{ll}
\mathit{S}(V_{11}) &=\dfrac{231 \pi  b^3 }{204800 a^9}\Big(785 a^3e_6+26 a^2 \sqrt{b}e_8-117 a be_4-13 b^{3/2}e_7\Big).
\end{array}
\end{equation}
So, in this case differently to the previous, assuming all the conditions under the parameters $e_3, e_2, e_1, e_9$ and $e_5$ the system \eqref{FormaCanonicaR-L-aneg} present a weak focus at the origin.
\end{proof}

\section{Melnikov function and stability of the saddle connections}\label{Secao-FuncoesMelnikov}

We now study the persistence and stability of the saddle connections of systems \eqref{FormaCanonicaR-L-apos} and \eqref{FormaCanonicaR-L-aneg}. We start analyzing on which (non generic) conditions the homoclinic and heteroclinic loops persist for these systems with $\varepsilon\neq0$. This is provided by the first order Melnikov functions associated to each system. Once we have established conditions to persistence, we will deal with the stability of such loops. More precisely, consider  $\mathit{A}_1^\varepsilon$ the unstable manifold of $p_2$ and $\mathit{A}_2^\varepsilon$ the stable manifold of $p_1$, see Figure \ref{MelnikovHeteroclinica}. Let $\mathit{A}^\varepsilon$ be the heteroclinic, homoclinic loop of systems \eqref{FormaCanonicaR-L-apos}, \eqref{FormaCanonicaR-L-aneg}, resp., and $p \in \mathit{A}^\varepsilon$. Consider
\[
n_1= \dfrac{(H_x(p), H_y(p))}{|| (H_x(p), H_y(p)) ||},
\]
then Melnikov function is
\[
d(\varepsilon, M_1)= - \langle n_1,  \overrightarrow{\mathit{A}_1^\varepsilon \mathit{A}_2^\varepsilon} \rangle,
\]
see Figure \ref{MelnikovHeteroclinica}, and for more details about the homoclinic and heteroclinic loops see \cite{Roussarie1986}.

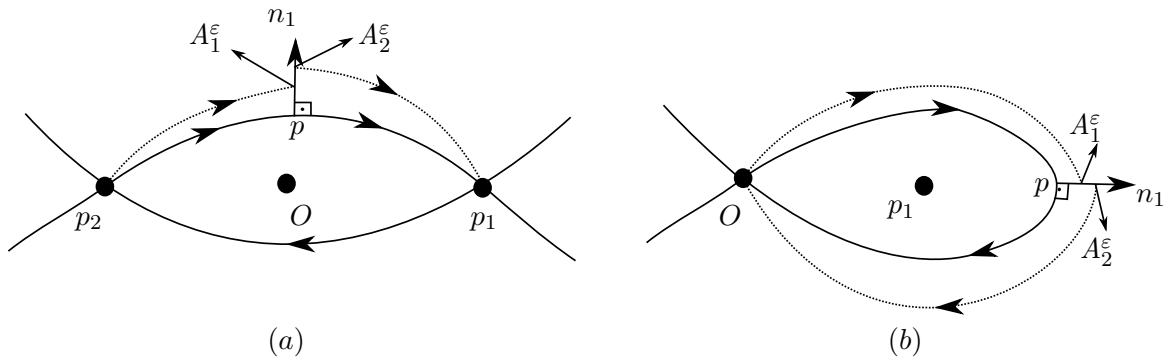
\begin{figure*}[h]
	\begin{center}
		\vspace{1cm}
		\begin{overpic}[width=15cm]{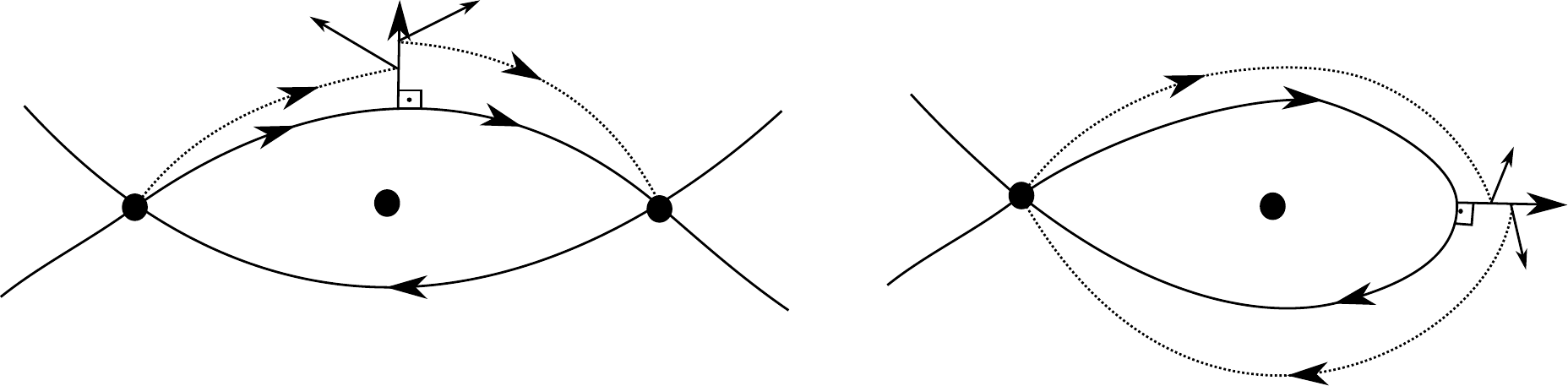}
			%\begin{overpic}[grid,tics=10,width=13cm]{MelnikovHeteroclinica.pdf}
			\put(25,16){$p$} \put(6,8){$p_2$}\put(25,8){$O$}\put(41,8){$p_1$}\put(23,26){$n_1$}\put(16,24){$\mathit{A}_1^\varepsilon$} \put(31,24){$\mathit{A}_2^\varepsilon$}
			\put(78,9){$p_1$}\put(63,8){$O$} \put(23,-3){$(a)$}\put(78,-3){$(b)$}\put(91,11){$p$} \put(100,10){$n_1$}\put(94,16){$\mathit{A}_1^\varepsilon$} \put(95,5){$\mathit{A}_2^\varepsilon$}
		\end{overpic}
	\end{center}
	\vspace{0.5cm}
	\caption{In figures $(a)$ and $(b)$ are presented the geometric approach of the construction of Melnikov functions of systems \eqref{FormaCanonicaR-L-apos} and \eqref{FormaCanonicaR-L-aneg}, respectively.}\label{MelnikovHeteroclinica}
\end{figure*}

\subsection{Persistence and stability of heteroclinic loop in system \eqref{FormaCanonicaR-L-apos}}\label{heteroclinica}

Denote by $Q_1(x,y)$ the perturbed part of system \eqref{FormaCanonicaR-L-apos}, that is, $Q_1(x,y)=(d_3 + d_2 x^2+ d_1 x^4 + d_4 y^2 + d_5 x^6+ d_6 y^4)y$. For system \eqref{FormaCanonicaR-L-apos} we have the following result:

\begin{lemma}\label{lema_funcao_melnikov}
	System \eqref{FormaCanonicaR-L-apos} has a heteroclinic loop $\mathit{A}_\varepsilon$ close to $\mathit{A}$ if and only if
	$$
	\begin{array}{ll}
d_2&=\varphi_1(d_1,d_3,d_4,d_5,d_6,\varepsilon)\\\\
	&=-\dfrac{12 d_4}{7}-\dfrac{5a^4d_5}{84b^2}+\dfrac{a^2(99d_1+160 d_6)}{462b}+\dfrac{10bd_3}{a^2}+\mathit{O}(\varepsilon).
	\end{array}
	$$
\end{lemma}
We write $\mathit{A}_\varepsilon=\{p_1\}\cup\mathit{A}_1^\varepsilon\cup\{p_2\}\cup\mathit{A}_2^\varepsilon$, where $\mathit{A}_{1,2}^\varepsilon$ denotes the orbits coming from $\mathit{A}_{1,2}$ for $\varepsilon\neq0$. In what follows we prove Lemma \ref{lema_funcao_melnikov} only for the arc $\mathit{A}_1$, because by symmetry $\mathit{A}_1$ is broken if and only if $\mathit{A}_2$ is broken.

\begin{proof}
	In order to obtain the conditions for the persistence of the heteroclinic loop of system \eqref{FormaCanonicaR-L-apos} with $\varepsilon=0$ it is sufficient to study the zeros of the first order Melnikov function $M_1$ associated to the heteroclinic loop. As done in \cite{Roussarie1986} this function depends on the parameters of the system and it is given by
	$$
	M_1=\displaystyle\int_{\mathit{A}_1} Q_1(x,y)\,dx,
	$$
	where we highlight that $M_1=M_1(a,b,d_1,\ldots,d_6)$. Using the expression of $\mathit{A}_1$ obtained in Subsection \ref{topological_structure} we get
	$$
	\begin{array}{rcl}
	M_1&=&\displaystyle\int_{\mathit{A}_1}(d_3 + d_2 x^2+ d_1 x^4 + d_4 y^2 + d_5 x^6+ d_6 y^4)y\,dx\vspace{0.2cm}\\
	&=&2\displaystyle\int_{0}^{\frac{a}{\sqrt{-{2b}}}} (d_3 + d_2 x^2+ d_1 x^4 + d_4 y(x)^2 + d_5 x^6+d_6 y(x)^4)y(x)\,dx
	\end{array}
	$$
	where $y(x)=-\sqrt{\frac{-b}{a^2}}x^2+\sqrt{\frac{a^2}{-4b}}$. Replacing $M_1$ into the expression of the integral we get
	$$
\displaystyle\int_{0}^{\frac{a}{\sqrt{-{2b}}}}-\dfrac{a^2+2bx^2}{2\sqrt{-a^2b}}(d_3+d_2x^2+d_1x^4+d_5 x^6-\dfrac{d_4}{4a^2b}(a^2+2bx^2)^2+\dfrac{d_6}{16a^4b^2}(a^2+2bx^2)^4)\,dx.
	$$
	Integrating we obtain
	$$
M_1=\dfrac{1}{13860\sqrt{2}b^4}(198 a^6 b d_1 - 924 a^4 b^2 d_2 + 9240 a^2 b^3 d_3-- 1584 a^4 b^2 d_4 - 55 a^8 d_5 + 320 a^6 b d_6)
	$$
	and finally solving this last equation with respect to the parameter $d_2$ we are done.
\end{proof}

Since system \eqref{sistema-R-Lgeneralizado} is analytic the heteroclinic loop is isolated, so one can study its stability.

\begin{lemma}\label{lema_estabilidade_laco}
	Assume that $d_2=\varphi_1$. The heteroclinic loop $\mathit{A}_\varepsilon$ is stable (respect. unstable) if the value
	$$
div(p_2)=\dfrac{1}{231b^3}(-22a^6d_5+a^4b(33d_1-40d_6)+198a^2b^2d_4-924d_3b^3)\varepsilon.
	$$
	in negative (respect. positive).
\end{lemma}

\begin{proof}
	The proof is straightforward from the computation of the divergence at the saddle point $p_2$.
\end{proof}

If the condition of Lemma \ref{lema_funcao_melnikov} occurs we have a simple loop. However if that divergence is zero in order to obtain the stability of the heteroclinic loop we must analyze higher orders of the return map, see \cite{Iliev}. In this direction we have the following result.

\begin{lemma}\label{lema_estabilidade_laco_2}
	Assume that $d_2=\varphi_1$ and $div(p_2)=0$, that is,
	$$
	d_1=\varphi_2(d_3,d_4,d_5,d_6)=\dfrac{40d_6}{33}+\dfrac{2a^2d_5}{3b}-\dfrac{6d_4b}{a^2}+\dfrac{28d_3b^2}{a^4}.
	$$
	Then the heteroclinic loop $\mathit{A}_\varepsilon$ is stable (resp. unstable) if the value $d_5$ is
	$$
	\varphi_3(d_3,d_4,d_6,\varepsilon)=\varepsilon\dfrac{6b(-100a^4d_6+231a^2bd_4+308d_3b^2)}{11a^6}
	$$
	is negative (resp. positive).	
\end{lemma}

\begin{proof}
	In \cite{Roussarie1986} we see that the stability of $\mathit{A}_\varepsilon$ is given by the sign of the expression
	$$
	\varepsilon\int_{\mathit{A}_\varepsilon}\frac{\partial Q_1}{\partial y}dt=\varepsilon\left[\int_{\mathit{A}}\frac{\partial Q_1}{\partial y}dt+O(\varepsilon)\right]
	$$
	Thus, since $\varepsilon\neq0$ is sufficiently small it is sufficient to study the sign of the right hand side of the last equality.
	
	Using the particular values of $d_1$ and $d_2$ we get
	$$
	\begin{array}{rcl}
	\dfrac{\partial Q_1}{\partial y}&=&\dfrac{d_3(a^2+2bx^2)(a^2+14bx^2)}{a^4}+\dfrac{x^2(a^2+2bx^2)(80bd_6+11d_5(a^2+6bx^2))}{132b^2}\vspace{0.2cm}\\
	&&+5d_6y^4+3d_4\left(-x^2-\dfrac{2bx^4}{a^2}+y^2\right).
	\end{array}
	$$
	Due to the symmetry and using the parametrization of arc $\mathit{A}_1$ obtained in Subsection \ref{topological_structure} we get the expression for $\varepsilon\displaystyle\int_{\mathit{A}_1^\varepsilon}\frac{\partial Q_1}{\partial y}dt$ is
	$$
	\varepsilon\dfrac{-11a^7d_5-600a^5bd_6+1386a^3b^2d_4+1848ab^3d_3}{3564\sqrt{-2b^7}}
	$$
	This last equation vanishes for $d_5=\varphi_3(d_3,d_4,d_6,\varepsilon)$, so we are done.
\end{proof}

In summary the function $\varphi_1$ controls the existence or not of the heteroclinic loop. If the loop remains for $\varepsilon\neq0$, the stability of it is determined by the sign of $\varphi_2$ if the loop is simple, or by the sign of $\varphi_2$ if not. Now we perform the same computations for the case $a<0$ and $b>0$.

\subsection{Persistence and stability of homoclinic loop in system \eqref{FormaCanonicaR-L-aneg}}\label{homoclinica}

\bigskip

We denoted by $Q_2(x,y)$ the perturbed part of system \eqref{FormaCanonicaR-L-aneg}, that is, $Q_2(x,y)=(e_3 + e_7 x+e_2 x^2+e_8 x^3+ e_1 x^4 +e_9 x^5+ e_4 y^2 + e_5 x^6+ e_6 y^4)y$. Analogously to Subsection \ref{heteroclinica} we study the Melnikov function and the stability of the homoclinic loop when it persists for $\varepsilon\neq0$.

\begin{lemma}\label{lema_funcao_melnikov_2}
System \eqref{FormaCanonicaR-L-aneg} has a homoclinic loop $\mathit{L}_\varepsilon$ close to $\mathit{L}$ if and only if
$$
\begin{array}{ll}
e_2&=\phi_1(e_1,e_3,e_4,e_5,e_6,e_7,e_8,e_9,\varepsilon)\\\\
	&=\dfrac{1}{5544a^2b^2(12\sqrt{2}-5\pi)}(-32\sqrt{2}(1155 b^3 e_3+198 a^2 b^2 e_4 + 32263 a^6 e_5 \\\\ &+ 5 a^4 b (1551 e_1 + 8 e_6) +1155 a b^{5/2} e_7 + 3927 a^3 b^{3/2} e_8 + 15675 a^5 \sqrt{b} e_9) \\\\
	&+3465 a (32 a^3 b e_1 + 134 a^5 e_5 + 4 b^{5/2} e_7+16 a^2 b^{3/2} e_8+65 a^4 \sqrt{b} e_9) \pi)+O(\varepsilon).
	\end{array}
	$$
\end{lemma}	

Now we call $\mathit{L}_\varepsilon=\mathit{L}_l^\varepsilon\cup\{q_1\}\cup\mathit{L}_r^\varepsilon$, where in this case $\mathit{L}_{l,r}^\varepsilon$ denotes the orbit arcs coming from $\mathit{L}_{l,r}$ for $\varepsilon\neq0$ sufficiently small.

\begin{proof}
	The conditions for the persistence of the homoclinic loop is provided by the zeros of the first order Melnikov function $M_2$ associated to the homoclinic loop, that is
$$
M_2=\displaystyle\int_{\mathit{L}_r} Q_2(x,y)\,dx.
$$	
	From the expression of the arc $\mathit{L}_r$ we have $M_2$ into the form
$$
\begin{array}{l}
	\displaystyle\int_{\frac{a}{\sqrt{b}}}^{\frac{a-a\sqrt{2}}{\sqrt{b}}}\dfrac{1}{2}\sqrt{\dfrac{(a-\sqrt{b}x)^2(a^2+2a\sqrt{b}x-bx^2)}{a^2b}}\Bigg(e_3+e_7 x+e_2x^2+e_8 x^3+e_1x^4+e_9 x^5\\\\ +e_5x^6+\dfrac{e_4}{4}\left(\dfrac{a^2}{b}-4x^2\right.	\left.+\dfrac{4\sqrt{b}}{a}x^3-\dfrac{b}{a^2}x^4\right)+\dfrac{e_6}{16}\left(\dfrac{a^2}{b}-4x^2+\dfrac{4\sqrt{b}}{a}x^3-\dfrac{b}{a^2}x^4\right)^2\Bigg).
	\end{array}
	$$	
	By integrating $M_2$ we get
	$$
	\begin{array}{ll}
	M_2&=\dfrac{a^2}{110880b^4}(32\sqrt{2}(1155 b^3 e_3 + 99 a^2 b^2 (21 e_2 + 2 e_4)+32263 a^6 e_5 + 5 a^4 b (1551 e_1 + 8 e_6) \\\\ &+ 1155 a b^{5/2} e_7+3927 a^3 b^{3/2} e_8 + 15675 a^5 \sqrt{b} e_9) - 3465 a (32 a^3 b e_1+8 a b^2 e_2 + 134 a^5 e_5\\\\ & + 4 b^{5/2} e_7 + 16 a^2 b^{3/2} e_8+65a^4\sqrt{b}e_9)\pi).
	\end{array}
	$$
	then solving this last equation with respect to the parameter $e_2$ the result follows.
\end{proof}

The proof of the next result is straightforward.

\begin{lemma}\label{lema_estabilidade_laco_21}
Assume that $e_2=\phi_1$. The homoclinic loop $\mathit{L}_\varepsilon$ is stable (resp. unstable) if the value
$$
\begin{array}{ll}
div(q_1)&=\dfrac{1}{5544 b^3(12 \sqrt{2} - 5\pi)}(-64 \sqrt{2}(-462b^3e_3 +99 a^2 b^2 e_4 + 15092 a^6 e_5 + a^4 b (2838 e_1 + 20 e_6) \\\\
&-462 a b^{5/2} e_7 + 924 a^3 b^{3/2} e_8 + 6798 a^5 \sqrt{b} e_9) \varepsilon+3465 (24 a^4 b e_1 - 8 b^3 e_3 + 126 a^6 e_5 - 4 a b^{5/2} e_7\\\\
	&+8 a^3 b^{3/2} e_8 + 57 a^5 \sqrt{b} e_9) \pi\varepsilon)
	\end{array}
	$$
	is negative (resp. positive).
\end{lemma}

\begin{lemma}\label{lema_estabilidade_laco_31}
Assume that $e_2=\phi_1$ and $div(q_1)=0$, that is,
$$
\begin{array}{ll}
e_1&=\phi_2(e_3,e_4,e_5,e_6,e_7,e_8,e_9)\\\\
&=\dfrac{1}{264a^4b(688 \sqrt{2} - 315\pi)}(64\sqrt{2}(462b^3 e_3-99 a^2 b^2 e_4 - 15092 a^6 e_5 - 20 a^4 b e_6 + 462 a b^{5/2} e_7 \\\\
&-924 a^3 b^{3/2} e_8 - 6798 a^5 \sqrt{b} e_9) + 3465 (126 a^6 e_5 +\sqrt{b} (-8 b^{5/2} e_3 - 4 a b^2 e_7 + 8 a^3 b e_8 + 57 a^5 e_9)) \pi)
\end{array}
$$
Then the homoclinic loop $\mathit{L}_\varepsilon$ is stable (resp. unstable) if the value
$$
\begin{array}{ll}
e_5&=\phi_3(e_3,e_4,e_6,e_7,e_8,e_9,\varepsilon)\\\\
&=\varepsilon K_1[-3 \sqrt{b} (-39424 b^{5/2} e_3 - 3252480 \sqrt{2} b^{5/2} e_3 - 1626240 a^2 b^{3/2} e_4 - 665280 \sqrt{2} a^2 b^{3/2} e_4 \\\\
&- 622080 a^4 \sqrt{b} e_6- 134400 \sqrt{2} a^4 \sqrt{b} e_6 - 39424 a b^2 e_7 - 73920 \sqrt{2} a b^2 e_7 +78848 a^3 b e_8 \\\\
&+ 147840 \sqrt{2} a^3 b e_8 - 630784 a^5 e_9 -123200 \sqrt{2} a^5 e_9+1034880 \sqrt{2} b^{5/2} e_3 \pi + 582120 \sqrt{2} a^2 b^{3/2} e_4 \pi \\\\
&+184800 \sqrt{2} a^4 \sqrt{b} e_6 \pi + 32340 \sqrt{2} a b^2 e_7 \pi -64680 \sqrt{2} a^3 b e_8 \pi - 121275 a^5 e_9 \pi + 266805 \sqrt{2} a^5 e_9 \pi)]
\end{array}
$$
is negative (resp. positive), where $K_1=\frac{1}{22 a^6 (-429824 + 40320 \sqrt{2} - 99225 \pi +154035 \sqrt{2}\pi)}$.
\end{lemma}

\begin{proof}
	As proceeded in the proof of Lemma \ref{lema_estabilidade_laco_2} we must study the sign of $\varepsilon\int_{\mathit{L}_1}\frac{\partial Q_2}{\partial y}dt$. For fixed values of $e_1$ and $e_2$ according with Lemmas \ref{lema_estabilidade_laco_21} and \ref{lema_estabilidade_laco_31}, we get
$$
\begin{array}{ll}
\dfrac{\partial Q_2}{\partial y}&=K_2 [(22 a^8 e_5 (35648 \sqrt{2} - 16065 \pi) x^2+495 a^7 \sqrt{b} e_9 (512 \sqrt{2} - 231 \pi) x^2 \\\\ &+ 1848 b^4 e_3 (16 \sqrt{2} - 15 \pi) x^4+924 a b^{7/2} e_7 (32 \sqrt{2} - 15 \pi) x^4+2 a^6 b x^2 (640 \sqrt{2} e_6\\\\
&- 539 e_5 (896 \sqrt{2} - 405 \pi) x^2)-1056 a^2 b^3 x^2 (-105 e_3 \pi+2 \sqrt{2} (100 e_3 + 3 e_4 x^2)) \\\\ &- 132 a^3 b^{5/2} x^2 (-105 \pi (7 e_7 + 2 e_8 x^2) + +64 \sqrt{2} (25 e_7 + 7 e_8 x^2)) - 33 a^5 b^{3/2} x^2 (-105 \pi (16 e_8 \\\\ &+  57 e_9 x^2)+128 \sqrt{2} (29 e_8 + 103 e_9 x^2)) + 8 a^4 b^2 (-10395 \pi (e_3+ e_7 x + e_8 x^3 + x^5 (e_9 + e_5 x)\\\\
& + 3 e_4 y^2 + 5 e_6 y^4) + 8 \sqrt{2} (2838 e_3+2838 e_7 x + 2838 e_8 x^3 - 20 e_6 x^4 + 2838 e_9 x^5 + 2838 e_8 x^3 \\\\ &- 20 e_6 x^4 + 2838 e_9 x^5 + 2838 e_5 x^6 + 14190 e_6 y^4 + 99 e_4 (x^2 + 86 y^2))))]
\end{array}
$$
being $K_2=1/(264 a^4 b^2 (688 \sqrt{2} - 315\pi))$.
	
As before, using symmetry and the parametrization of $\mathit{L}_1$ obtained in Subsection \ref{topological_structure} we obtain
$$
\begin{array}{ll}
\displaystyle\int_{\mathit{L}_1}\frac{\partial Q_2}{\partial y}dt=
&-K_3[a (27720 a^2 b^2 e_4 (-8 (22 + 9 \sqrt{2}) + 63 \sqrt{2} \pi) +29568 b^3 e_3 (-4 - 330 \sqrt{2} \\\\ &+ 105 \sqrt{2}\pi)   + 924 a b^{5/2} e_7 (-16 (8 + 15 \sqrt{2})+105 \sqrt{2} \pi)-1848 a^3 b^{3/2} e_8 (-16 (8 \\\\
&+ 15 \sqrt{2})+105 \sqrt{2}\pi)+1440 a^4 b e_6 (-8 (162 + 35 \sqrt{2}) + 385 \sqrt{2}\pi)\\\\
&+231 a^5 \sqrt{b} e_9 (-64 (128 + 25 \sqrt{2})+315 (-5 + 11 \sqrt{2}\pi)+22 a^6 e_5 (128 (-3358+ 315 \sqrt{2})\\\\
&+ 945 (-105 + 163 \sqrt{2}) \pi)) \varepsilon].
\end{array}
$$

%$$
%\begin{array}{ll}
%&+ 945 (-105 + 163 \sqrt{2}) \pi)) \varepsilon].
%\end{array}
%$$
where $K_3=1/(13860 b^{7/2} (688 \sqrt{2} - 315\pi))$. The last equation vanishes for $e_5=\phi_3(e_3,e_4,e_6,e_7,e_8,$ $e_9,\varepsilon)$, so we are done.
\end{proof}

Function $\phi_1$ has the same role than the function $\varphi$, but now it controls the existence and stability of the double homoclinic loop.

\section{Proof of the main result}\label{Secao-Prova}

In this section we prove the main result of the paper. Since Theorem \ref{teorema-apos} contemplates several cases for different values of $s$, $m$ and $k$, we split the proof in some lemmas, see next subsection. Once the lemmas are stated and proved, the proof of the main theorem follows quite straightforwardly. However we emphasize that lemmas are also important by themselves because they explicitly exhibit the approach of choosing suitable perturbations of the parameters to obtain limit cycles, as described in Section \ref{Secao-Resultados}. This is important because such calculations sometimes are neglected or exhibited for weak focus of low multiplicity as occurs in Hopf Theorem.

\subsection{Technical lemmas}\label{SessaoLemas}

In this subsection we state and prove some lemmas where we explicitly apply Bautin's algorithm, see \cite{Bautin,LivroLlibre}, combined with Melnikov techniques. That is actually the very essence of the proof of Theorem \ref{teorema-apos}. We start exhibiting Bautin's algorithm for obtaining five limit cycles which bifurcate from the center at the origin.

\begin{lemma}\label{lema-(5,0),apos}
	There exists at least five limit cycles bifurcating from the origin for system \eqref{FormaCanonicaR-L-apos}.
\end{lemma}

\begin{proof}
	In the proof of Lemma \ref{lema-coeficientesLyapunov-apos} under the hypothesis $d_3=0$ we obtained that: if the parameters $d_i$'s satisfy $d_2=-3 d_4,\, d_1 = -\frac{6 b d_4}{a^2}-5 d_6, \,d_5=\frac{6 d_4^3 \varepsilon ^2}{5}-\frac{16 b d_6}{a^2},\, d_6=\frac{12 a^2 b d_4^3 \varepsilon ^2}{5 \left(7 b^2-3 a^4 d_4^2 \varepsilon ^2\right)}$ and $d_4=0$, then system \eqref{FormaCanonicaR-L-apos} is Hamiltonian and therefore the origin is a center. Our objective here is to consider small perturbations in these values of parameters to obtain a variation of the signs of the Lyapunov constants and the stability of the origin.
	
	First we consider the expression of $V_{11}$ given in \eqref{V11-apos}. Note that if $d_4>0, d_4<0$ then the origin is a weak repeller focus, weak attractor focus, resp.. Without loss of generality, we assume that $d_4>0$.
	
	Consider now the expression of $V_{9}$ given in \eqref{V9-apos} replacing $d_6$ by $d_{62} \varepsilon ^2+ d_{63} \varepsilon ^3$, with $d_{62}=\frac{12 a^2 d_4^3}{35 b}$. Then we obtain the new expression of
	\[
	V_9=-\dfrac{21 \pi  b^2 d_{63} \varepsilon ^4}{32 a^4}+\frac{27 \pi  a^2 d_{4}^5 \varepsilon ^5}{280 b}+\frac{9}{32} \pi  d_{4}^2 d_{63} \varepsilon ^6.
	\]
	So the sign of $V_9$ is given by the sign of $d_{63}$. Then $|V_{9}| \ll  |V_{11}|$, and if $d_4>0$ and $d_{63}>0$ then $V_{9} \cdot V_{11}<0$. In this way we obtain at least one limit cycle bifurcating from the origin under these assumptions.
	
	Consider the expression of $V_7$ given in \eqref{V7-apos}, replacing $d_6$ by $d_{62} \varepsilon ^2+ d_{63} \varepsilon ^3$ and $d_5$ by $d_{52} \varepsilon ^2+d_{53} \varepsilon ^3 + d_{54} \varepsilon ^4$ with $d_{62}=\frac{12 a^2 d_4^3}{35 b}, d_{52}=-\frac{1}{7} \left(30 d_{4}^3\right)$ and $d_{53}=-\frac{16 b d_{63}}{a^2}$. Therefore we obtain
	\[
	V_7=\dfrac{-5}{64}  \pi  d_{54} \varepsilon ^5.
	\]
	As in the previous case we obtain $|V_7| \ll |V_{9}| \ll  |V_{11}|$, and if $d_4>0, d_{63}>0$ and $d_{54}<0$ then $V_{9} \cdot V_{11}<0$ and $V_{7} \cdot V_{9}<0$. So we obtain that at least two limit cycles bifurcating from the origin.
	
	Working similarly we consider the expression $V_5$ given in \eqref{V5-apos}, replacing $d_6$ by $d_{62} \varepsilon ^2+ d_{63} \varepsilon ^3, d_5$ by $d_{52} \varepsilon ^2+d_{53} \varepsilon ^3 + d_{54} \varepsilon ^4$ and $d_1$ by $d_{10}+ d_{11} \varepsilon+ d_{12} \varepsilon ^2 + d_{13} \varepsilon ^3 + d_{14} \varepsilon ^4 + d_{15} \varepsilon^5$ with $d_{62}=\frac{12 a^2 d_4^3}{35 b}, d_{52}=-\frac{1}{7} \left(30 d_{4}^3\right), d_{53}=-\frac{16 b d_{63}}{a^2}, d_{10} = -\frac{6 b d_{4}}{a^2},  d_{11} = 0, d_{12} =-\frac{12 a^2 d_{4}^3}{7 b},  d_{13} = -5 d_{63}$ and $d_{14} = 0$, we get
	\[
	V_5=-\dfrac{1}{8} \pi  d_{15} \varepsilon ^6.
	\]
	So we obtain that $|V_5| \ll |V_7| \ll |V_{9}| \ll  |V_{11}|$, and if $d_4>0, d_{63}>0, d_{54}<0$ and $d_{15} > 0$ then $V_{9} \cdot V_{11}<0, V_{7} \cdot V_{9}<0$ and $V_{5} \cdot V_{7}<0$. So we obtain that at least three limit cycles bifurcating from the origin.
	
	Consider now the expression of $V_3$ given in \eqref{V3-apos}, replacing $d_6$ by $d_{62} \varepsilon ^2+ d_{63} \varepsilon ^3, d_5$ by $d_{52} \varepsilon ^2+d_{53} \varepsilon ^3 + d_{54} \varepsilon ^4, d_1$ by $d_{10}+ d_{11} \varepsilon+ d_{12} \varepsilon ^2 + d_{13} \varepsilon ^3 + d_{14} \varepsilon ^4 + d_{15} \varepsilon^5$ and $d_2$ by $d_{20}+ d_{21} \varepsilon+ d_{22} \varepsilon ^2 + d_{23} \varepsilon ^3 + d_{24} \varepsilon ^4 + d_{25} \varepsilon^5 + d_{26} \varepsilon^6$ with $d_{62}=\frac{12 a^2 d_4^3}{35 b}, d_{52}=-\frac{1}{7} \left(30 d_{4}^3\right), d_{53}=-\frac{16 b d_{63}}{a^2}, d_{10} = -\frac{6 b d_{4}}{a^2},  d_{11} = 0, d_{12} =-\frac{12 a^2 d_{4}^3}{7 b},  d_{13} = -5 d_{63}, d_{14} = 0, d_{20} = -3 d_{4}, d_{21}=0, d_{22}=0, d_{23}=0, d_{24}=0$ and $d_{25}=0$, we obtain
	\[
	V_3=-\dfrac{1}{4} \pi  d_{26} \varepsilon ^7.
	\]
	Therefore we get that $|V_3| \ll |V_5| \ll |V_7| \ll |V_{9}| \ll  |V_{11}|$, and if $d_4>0, d_{63}>0, d_{54}<0, d_{15} > 0$ and $d_{26} < 0$ then $V_{9} \cdot V_{11}<0, V_{7} \cdot V_{9}<0, V_{5} \cdot V_{7}<0$ and $V_{3} \cdot V_{5}<0$. So we obtain that at least four limit cycles bifurcating from the origin.
	
	Finally, replacing $d_6$ by $d_{62} \varepsilon ^2+ d_{63} \varepsilon ^3, d_5$ by $d_{52} \varepsilon ^2+d_{53} \varepsilon ^3 + d_{54} \varepsilon ^4, d_1$ by $d_{10}+ d_{11} \varepsilon+ d_{12} \varepsilon ^2 + d_{13} \varepsilon ^3 + d_{14} \varepsilon ^4 + d_{15} \varepsilon^5, d_2$ by $d_{20}+ d_{21} \varepsilon+ d_{22} \varepsilon ^2 + d_{23} \varepsilon ^3 + d_{24} \varepsilon ^4 + d_{25} \varepsilon^5 + d_{26} \varepsilon^6$ and $d_3$ by $d_{37} \varepsilon^7$, with $d_{62}=\frac{12 a^2 d_4^3}{35 b}, d_{52}=-\frac{1}{7} \left(30 d_{4}^3\right), d_{53}=-\frac{16 b d_{63}}{a^2}, d_{10} = -\frac{6 b d_{4}}{a^2},  d_{11} = 0, d_{12} =-\frac{12 a^2 d_{4}^3}{7 b},  d_{13} = -5 d_{63}, d_{14} = 0, d_{20} = -3 d_{4}, d_{21}=0, d_{22}=0, d_{23}=0, d_{24}=0$ and $d_{25}=0$, we get that the origin is a stable focus if $d_{37}<0$ and unstable focus if $d_{37}>0$.
	
	Hence we obtain that $|V_3| \ll |V_5| \ll |V_7| \ll |V_{9}| \ll  |V_{11}|$, and if $d_4>0, d_{63}>0, d_{54}<0, d_{15} > 0, d_{26} < 0$ and $d_{37}>0$ then $V_{9} \cdot V_{11}<0, V_{7} \cdot V_{9}<0, V_{5} \cdot V_{7}<0$ and $V_{3} \cdot V_{5}<0$. In this way we obtain one more limit cycle bifurcating from the origin. Therefore we get at least five limit cycles bifurcating from the origin, concluding the proof.
%	
%	
%	\[
%	V_1=1+d_{37}\pi\varepsilon^8+\mathit{O}(\varepsilon^{16}).
%	\]
%	%V_1=\frac{1}{2} \pi ^2 d_{37}^2 \varepsilon ^{16}+\pi  d_{37} \varepsilon ^8
%	So we obtain that $|V_3| \ll |V_5| \ll |V_7| \ll |V_{9}| \ll  |V_{11}|$, and if $d_4>0, d_{63}>0, d_{54}<0, d_{15} > 0, d_{26} < 0$ and $d_{37}<0$ then $V_{9} \cdot V_{11}<0, V_{7} \cdot V_{9}<0, V_{5} \cdot V_{7}<0$ and $V_{3} \cdot V_{5}<0$. Moreover, assuming $d_{37}<0$ then $V_1<1$ and the origin is an attractor. In this way we obtain one more limit cycle bifurcating from the origin. Therefore we get at least five limit cycles bifurcating from the origin, concluding the proof.
\end{proof}

\begin{remark}\label{multiplicidade}
	One could argue about the existence of more than five limit cycles inside the heteroclinic loop. Although that situation could be realizable, usually the multiplicity of the weak focus determines an upper bound for the number of limit cycles inside the loop. For instance, in the previous lemma it is easy to see that the first limit cycle bifurcation from the origin, call $\gamma$, is unstable. Moreover we are able to compare this information with the persistence or not of the heteroclinic loop. Indeed by fixing the values of parameters for which we obtain five limit cycles, the expression of $M_1$ from Subsection \ref{heteroclinica} we get
	$$
	M_1=\dfrac{a^8d_4^3}{2310\sqrt{2}b^4}\varepsilon^2+\mathit{O}(\varepsilon^3).
	$$
	Since we are assuming $d_4>0$ we get that $M_1>0$. From Subsection \ref{heteroclinica} the orbit leaving $p_1$ goes away from $\gamma$, so by using the Poincar\'e-Bendixson Theorem a convenient annular region it follows that system \eqref{FormaCanonicaR-L-apos} has either no limit cycles between $\gamma$ and the saddle point $p_1$ or it appears in pairs. So we conjecture that five is an upper bound for the number of limit cycles on the region located between the weak focus and the saddle point.
\end{remark}

\begin{lemma}\label{lema-(2,3),apos}
	There exist a suitable choice of parameters such that system \eqref{FormaCanonicaR-L-apos} has simultaneously 3 limit cycles bifurcating from the heteroclinic loop and 2 limit cycles bifurcating from the origin.
\end{lemma}

\begin{proof}
	As we see in Subsection \ref{heteroclinica} system \eqref{FormaCanonicaR-L-apos} has a heteroclinic loop $\mathit{A}_1^\varepsilon$ if $d_2=-\frac{12 d_4}{7}-\frac{5a^4d_5}{84b^2}+\frac{a^2(99d_1+160 d_6)}{462b}+\frac{10bd_3}{a^2}+\mathit{O}(\varepsilon)$. Moreover, if $d_1=\frac{40d_6}{33}+\frac{2a^2d_5}{3b}-\frac{6d_4b}{a^2}+\frac{28d_3b^2}{a^4}$ then the stability of that loop is established by the sign of $d_5=6b(-100a^4d_6+231a^2bd_4+308d_3b^2)/11a^6$. The first part of the proof consists in changing the stability of the loop to get 2 limit cycles, and finally to destroy the loop in order to obtain one more limit cycle. We start assuming $d_5<0$ so $\mathit{A}_1^\varepsilon$ is stable. Now we write $d_5=d_{50}+\varepsilon d_{51}$ with $d_{50}=(1848 b^3 d_3 + 1386 a^2 b^2 d_4 - 600 a^4 b d_6)/11a^6$ so
	\begin{equation}\label{M1}
	M_1=\displaystyle\int_{\mathit{A}_1}\frac{\partial Q_1}{\partial y}dt=-\dfrac{a^7d_{51}}{315\sqrt{-2b^7}}\varepsilon^2.
	\end{equation}
	Consequently since $a>0$ the loop $\mathit{A}_1^\varepsilon$ now is defined by the sign of $d_{51}$. Indeed, if $d_{51}<0$ then the loop changes its stability from stable to unstable and by applying Poincar\'e-Bendixson Theorem a stable limit cycle emerges.
	
	In order to obtain a new limit cycle we change again the stability of $\mathit{A}_1^\varepsilon$. For doing this, consider the expression of $div(p_2)=\frac{\partial Q_1}{\partial y}$ replacing $d_1$ by $d_{10}+\varepsilon d_{11}+\varepsilon^2 d_{12}$ with $d_{11}=2a^2 d_{51}/3b$ and $d_{10}=(66b(70b d_3+39a^2 d_4)-1160a^4 d_6)/33a^4$, we get
\begin{equation}\label{div}
div(p_2)=\dfrac{a^2d_{12}}{7b^2}\varepsilon^3.
\end{equation}
Therefore by choosing $d_{12}<0$ the heteroclinic loop changes stability from unstable to stable and again, by the Poincar\'e-Bendixson theorem, a limit cycle bifurcates from $\mathit{A}_1^\varepsilon$. More precisely, that limit cycles is unstable.
	Finally to obtain the third limit cycle we choose the parameters so that the loop $\mathit{A}_1^\varepsilon$ is broken. We do that by taking $d_2=d_{20}+\varepsilon d_{21}+\varepsilon^2 d_{22}+\varepsilon^3 d_{23}$ where now we choose $d_{22}=3a^2 d_{12}/14b$, $d_{21}=a^4 d_{51}/12b^2$ and $d_{20}=(1980a^2b^2 d_3+495 a^2 b d_4-260 a^4 b d_6)/66a^2 b$. Then we obtain
	$$
	M_1=-\dfrac{a^4d_{23}}{15\sqrt{2}b^2}\varepsilon^3.
	$$
	We see that assuming $d_{23}\neq0$ the loop is broken. Moreover, from Section \ref{Secao-FuncoesMelnikov} and since the second limit cycle is unstable, by choosing $d_{23}>0$ we get $M_1<0$. Consequently, by the Poincar\'e-Bendixson Theorem we obtain a third limit cycle which is stable. Therefore we obtain three limit cycles by using Melnikov method. Moreover, for $\varepsilon$ positive and sufficiently small we have that the expression of equations \eqref{M1} and \eqref{div} have opposite sign and $|div(p_2)|\ll|M_1|$. It means that the change of stability of $\mathit{A}_1^\varepsilon$ are only local so we can apply the Poincar\'e-Bendixson Theorem to obtain the limit cycles.
	
	The proof of the lemma follows by applying the same approach used in Lemma \ref{lema-(5,0),apos} having now fixed the values $d_5=d_{50}+\varepsilon d_{51}$, $d_1=d_{10}+\varepsilon d_{11}+\varepsilon^2 d_{12}$ and $d_2=d_{20}+\varepsilon d_{21}+\varepsilon^2 d_{22}+\varepsilon^3 d_{23}$. That is, it can be obtained two more limit cycles but now bifurcating from the origin of system \eqref{FormaCanonicaR-L-apos} so we get 5 limit cycles for such a system. The simultaneity occurs because at each step the obtained limit cycles are hyperbolic so they remain by assuming perturbations of the parameters involving higher orders of $\varepsilon$.
\end{proof}

\begin{lemma}\label{lema-(s,m),apos}
	There exist a suitable choice of parameters such that system \eqref{FormaCanonicaR-L-apos} has $s$ limit cycles bifurcating from the origin and $m$ limit cycles bifurcating from the heteroclinic loop with $s\in\{0,1,3,4,5\}$, $m\in\{0,1,2\}$ and $s+m\leq5$.
\end{lemma}

\begin{proof}
	The proof of this Lemma \ref{lema-(s,m),apos} is straightforward by using the same construction of Lemmas \ref{lema-(5,0),apos} and \ref{lema-(2,3),apos}.
\end{proof}

We now state similar results but concerning system \eqref{FormaCanonicaR-L-aneg}. Again we obtain limit cycles bifurcating from the center and from the homoclinic loops considering convenient values of the parameters of system \eqref{FormaCanonicaR-L-aneg}. In addition we also consider some symmetry in system \eqref{FormaCanonicaR-L-aneg} so that limit cycles may emerge in pairs.

\begin{lemma}\label{lema-(5,0),aneg}
	There exists at least five limit cycles bifurcating from the origin of system \eqref{FormaCanonicaR-L-aneg}.
\end{lemma}

\begin{proof}
In the proof of Lemma \ref{lema-coeficientesLyapunov-aneg} under the hypothesis $e_3=0$ we obtain that: if the parameters $e_i$ satisfies the conditions
\[
\begin{array}{ll}
e_2&=-\dfrac{3 \left(2 ae_4+\sqrt{b}e_7\right)}{2 a},\\\\
e_1  &= \dfrac{ 1 }{12 a^3} \Big( 8 a^3e_4e_7^2 \varepsilon ^2-60 a^3e_6-30 a^2 \sqrt{b}e_8+117 a be_4+15 b^{3/2}e_7\Big),\\\\
e_9   &=\frac{  1} {105 a^4 \sqrt{b}} \Big( 36 a^5e_4^3 \varepsilon ^2+48 a^5e_4e_7e_8 \varepsilon ^2-30 a^5e_5-20 a^5e_6e_7^2 \varepsilon ^2 -252 a^4 \sqrt{b}e_4^2e_7 \varepsilon ^2-158 a^3 be_4e_7^2 \varepsilon ^2\\\\
&+1005 a^3 be_6+210 a^2 b^{3/2}e_8-945 a b^2e_4-105 b^{5/2}e_7\Big),
\end{array}
\]
and
\[
\begin{array}{ll}
e_5  &=\dfrac{1}{120 a^5 \left(848 a^3e_4e_7 \varepsilon ^2+7875 b^{3/2}\right)}
\Big( 122112 a^8e_4^4e_7 \varepsilon ^4+162816 a^8e_4^2e_7^2e_8 \varepsilon ^4-67840 a^8e_4e_6e_7^3 \varepsilon ^4\\\\ &-592704 a^7 \sqrt{b}e_4^3e_7^2 \varepsilon ^4+907200 a^7 \sqrt{b}e_4^2e_6 \varepsilon ^2+110880 a^7 \sqrt{b}e_4e_8^2 \varepsilon ^2+40320 a^7 \sqrt{b}e_6e_7^4 \varepsilon ^4\\\\ &+161280 a^7 \sqrt{b}e_6e_7e_8 \varepsilon ^2+156224 a^6 be_4^2e_7^3 \varepsilon ^4-1239840 a^6 be_4^2e_8 \varepsilon ^2-2356800 a^6 be_4e_6e_7 \varepsilon ^2\\\\ &+2426760 a^5 b^{3/2}e_4^3 \varepsilon ^2-1782480 a^5 b^{3/2}e_4e_7e_8 \varepsilon ^2+587160 a^5 b^{3/2}e_6e_7^2 \varepsilon ^2+9124920 a^4 b^2e_4^2e_7 \varepsilon ^2\\\\ &+1912470 a^3 b^{5/2}e_4e_7^2 \varepsilon ^2-7867125 a^3 b^{5/2}e_6 -727650 a^2 b^3e_8+3274425 a b^{7/2}e_4+363825 b^4e_7\Big)
\end{array}
\]
then the system \eqref{FormaCanonicaR-L-aneg} presenting a weak focus at the origin. Similarly we did in the proof of Lemma \ref{lema-(5,0),apos}, we will consider small perturbations in these values of parameters to obtain a variation of the sign of the Lyapunov constants and the stability of the origin.

The signal of $V_{11}$ is given by the sign of $\mathit{S}(V_{11})$ presented in  \eqref{sinal-V11-CasoAneg}. Therefore, if
\[
\begin{array}{lrl}
e_7& >&\dfrac{a \left(785 a^2 e_6 + 26 a \sqrt{b} e_8-117 b e_4\right)}{13 b^{3/2}},\\\\ e_7 &< &\dfrac{a \left(785 a^2 e_6 + 26 a \sqrt{b} e_8-117 b e_4\right)}{13 b^{3/2}},
\end{array}
\]
resp., then the origin is a weak repeller focus, weak attractor focus, resp.. Without loss of generality we assume that $e_7>\dfrac{a \left(785 a^2 e_6 + 26 a \sqrt{b} e_8-117 b e_4\right)}{13 b^{3/2}}$. Consider the expression of $V_{9}$ given in \eqref{V9-aneg} replacing $e_5$ by $e_{50}+ e_{51} \varepsilon$, with
\[
e_{50}=    \frac{b \left(-1665 a^3 e_6-154 a^2 \sqrt{b} e_8+693 a b e_4+77 b^{3/2} e_7\right)}{200 a^5},
\]
we obtain
\[
\begin{array}{ll}
V_9  &=\dfrac{  1} {50400 a \sqrt{b}} \Big( \pi  e_7 \varepsilon ^5 \Big(212 a^2 e_4 \Big(9 e_4^3+12 e_4 e_7 e_8-5 e_6 e_7^2\Big)+63 a \sqrt{b} e_7 \Big(10 e_6 e_7^2-147 e_4^3\Big)\\\\
&+2441 b e_4^2 e_7^2\Big)\Big) - \dfrac{75 \pi  b e_{51} \varepsilon ^2}{256 a^2}+\dfrac{ 1} {384000 a^4} \Big(  \pi  \varepsilon ^3 \Big(1200 a^4 \Big(90 e_4^2 e_6+11 e_4 e_8^2+16 e_6 e_7 e_8\Big)\\\\
&-40 a^3 \sqrt{b} e_4 (3690 e_4 e_8+4493 e_6 e_7) + 12 a^2 b \Big(24075 e_4^3-16906 e_4 e_7 e_8+5825 e_6 e_7^2\Big)\\\\
&+1044324 a b^{3/2} e_4^2 e_7+223011 b^2 e_4 e_7^2\Big) \Big)-\dfrac{53 \pi  a e_4 e_{51} e_7 \varepsilon ^4}{1680 \sqrt{b}}.
\end{array}
\]
So the sign of $V_9$ is given by the sign of $e_{51}$. Then $|V_{9}| \ll  |V_{11}|$ and if $e_{51}>0$ and
\[
e_7>\dfrac{a \left(785 a^2 e_6 + 26 a \sqrt{b} e_8-117 b e_4\right)}{13 b^{3/2}}
\]
then $V_{9} \cdot V_{11}<0$. Therefore we conclude that at least one limit cycle bifurcating from the origin with these assumptions.
	
Consider the expression of $V_7$ given in \eqref{V7-aneg}, replacing $e_9$ by $e_{90}+e_{91} \varepsilon + e_{92} \varepsilon^2$ and $e_5$ by $e_{50}+ e_{51} \varepsilon$, with
\[
\begin{array}{ll}
e_{50}   &=  \dfrac{b \left(-1665 a^3 e_6-154 a^2 \sqrt{b} e_8+693 a b e_4+77 b^{3/2} e_7\right)}{200 a^5},\\\\
e_{90}   &=\dfrac{1195 a^3 \sqrt{b} e_6+222 a^2 b e_8-999 a b^{3/2} e_4-111 b^2 e_7}{100 a^4}
\end{array}
\]
and $ e_{91}=-\frac{2 a e_{51}}{7 \sqrt{b}}$, we obtain
\[
\begin{array}{ll}
V_7   &=\dfrac{1 }{384 a^2} \Big( \pi  \varepsilon ^3 \Big(4 a^2 \Big(9 e_4^3+12 e_4 e_7 e_8-5 e_6 e_7^2\Big)- 21 a \sqrt{b} \Big(12 e_4^2 e_7+5 e_{92}\Big)-158 b e_4 e_7^2\Big)\Big) .
\end{array}
\]
Then $|V_7| \ll |V_{9}| \ll  |V_{11}|$, and if $e_{51} > 0$,
\[
\begin{array}{ll}
e_7         & > \dfrac{a \left(785 a^2 e_6 + 26 a \sqrt{b} e_8-117 b e_4\right)}{13 b^{3/2}},\\\\
e_{92} &> \dfrac{2}{105} \Big(\frac{2 a \left(9 e_4^3+12 e_4 e_7 e_8-5 e_6 e_7^2\right)}{\sqrt{b}}-\frac{79 \sqrt{b} e_4 e_7^2}{a}-126 e_4^2 e_7\Big)
\end{array}
\]
then $V_{9} \cdot V_{11}<0$ and $V_{7} \cdot V_{9}<0$. So we obtain that at least two limit cycles bifurcating from the origin.
	
Consider the expression of $V_5$ given in \eqref{V5-aneg}, replacing $e_9$ by $e_{90}+e_{91} \varepsilon + e_{92} \varepsilon^2$ and $e_5$ by $e_{50}+ e_{51} \varepsilon$ and $e_1$ by $e_{10}+ e_{12} \varepsilon^2+ e_{13} \varepsilon^3$, with
$e_{50}=    \frac{b}{200 a^5} \Big(-1665 a^3 e_6-154 a^2 \sqrt{b} e_8+693 a b e_4+77 b^{3/2} e_7\Big), e_{90}=\frac{1}{100 a^4} \Big( 1195 a^3 \sqrt{b} e_6$ $+222 a^2 b e_8-999 a b^{3/2} e_4-111 b^2 e_7\Big), e_{91}=-\frac{2 a e_{51}}{7 \sqrt{b}}, e_{12}=\frac{2 e_4 e_7^2}{3}, e_{10}=$ \linebreak$\frac{-20 a^3 e_6-10 a^2 \sqrt{b} e_8+39 a b e_4+5 b^{3/2} e_7}{4 a^3}$, therefore
\[
V_5=-\dfrac{1}{8} \pi  e_{13} \varepsilon ^4.
\]
So we obtain that $|V_5| \ll |V_7| \ll |V_{9}| \ll  |V_{11}|$ and if $e_7 > \dfrac{a \left(785 a^2 e_6 + 26 a \sqrt{b} e_8-117 b e_4\right)}{13 b^{3/2}}, e_{51} > 0, e_{92} > \frac{2}{105} \left(\frac{2 a \left(9 e_4^3+12 e_4 e_7 e_8-5 e_6 e_7^2\right)}{\sqrt{b}}-\frac{79 \sqrt{b} e_4 e_7^2}{a}-126 e_4^2 e_7\right)$ and $e_{13} > 0$ then $V_{9} \cdot V_{11}<0, V_{7} \cdot V_{9}<0$ and $V_{5} \cdot V_{7}<0$, then we obtain that at least three limit cycles bifurcating from the origin.

Now consider the expression of $V_3$ given in \eqref{V3-aneg}, replacing $e_9$ by $e_{90}+e_{91} \varepsilon + e_{92} \varepsilon^2$ and $e_5$ by $e_{50}+ e_{51} \varepsilon$ and $e_{10}+ e_{12} \varepsilon^2+ e_{13} \varepsilon^3$ and $e_2$ by $e_{20}+ e_{24} \varepsilon^4$, with $e_{50}=    \frac{b }{200 a^5}(-1665 a^3 e_6-154 a^2 \sqrt{b} e_8+693 a b e_4+$ $77 b^{3/2} e_7),$ $e_{90}=$ $\frac{1}{100 a^4}\Big( 1195 a^3 \sqrt{b} e_6+222 a^2 b e_8-999 a b^{3/2} e_4 -$ $111 b^2 e_7\Big),$ $ e_{91}=$$-\frac{2 a \text{e51}}{7 \sqrt{b}}, e_{10}=$\linebreak$\frac{ 1}{4 a^3}\Big(  -20 a^3 e_6-10 a^2 \sqrt{b} e_8+39 a b e_4+5 b^{3/2} e_7\Big), e_{12}=$ $\frac{2 e_4 e_7^2}{3}, e_{20}=-\frac{3 \sqrt{b} e_7}{2 a}-3 e_4$,
we obtain
\[
V_3=-\dfrac{1}{4} \pi e_{24} \varepsilon ^5.
\]
Therefore we get that $|V_3| \ll |V_5| \ll |V_7| \ll |V_{9}| \ll  |V_{11}|$, and if $e_7 > \frac{a \left(785 a^2 e_6 + 26 a \sqrt{b} e_8-117 b e_4\right)}{13 b^{3/2}}, $\linebreak $e_{51} > 0, e_{92} > \frac{2}{105} \Big(\frac{2 a \left(9 e_4^3+12 e_4 e_7 e_8-5 e_6 e_7^2\right)}{\sqrt{b}}-\frac{79 \sqrt{b} e_4 e_7^2}{a}$ $-126 e_4^2 e_7\Big), e_{13} > 0$ and $e_{24} < 0$ then $V_{9} \cdot V_{11}<0, V_{7} \cdot V_{9}<0, V_{5} \cdot V_{7}<0$ and $V_{3} \cdot V_{5}<0$. So we obtain that at least four limit cycles bifurcating from the origin.
	
Finally, replacing $e_9$ by $e_{90}+e_{91} \varepsilon + e_{92} \varepsilon^2$ and $e_5$ by $e_{50}+ e_{51} \varepsilon$ and $e_{10}+ e_{12} \varepsilon^2+ e_{13} \varepsilon^3$ and $e_2$ by $e_{20}+ e_{24} \varepsilon^4$ and $e_3$ by $e_{35} \varepsilon^5$, with $e_{50}=    \frac{b }{200 a^5} \Big(-1665 a^3 e_6-154 a^2 \sqrt{b} e_8+693 a b e_4+77 b^{3/2} e_7\Big),e_{90}=\frac{1195 a^3 \sqrt{b} e_6+222 a^2 b e_8-999 a b^{3/2} e_4-111 b^2 e_7}{100 a^4},$ $e_{91}=-\frac{2 a \text{e51}}{7 \sqrt{b}}, e_{10}=\frac{-20 a^3 e_6-10 a^2 \sqrt{b} e_8+39 a b e_4+5 b^{3/2} e_7}{4 a^3}, $ $e_{12}=\frac{2 e_4 e_7^2}{3}$, $e_{20}=-\frac{3 \sqrt{b} e_7}{2 a}-3 e_4$, we get that the origin is a stable focus if $e_{35}<0$ and unstable focus if $e_{35}>0$.

Hence we obtain that $|V_3| \ll |V_5| \ll |V_7| \ll |V_{9}| $ $\ll  |V_{11}|$, and if
$e_{51} > 0, e_{13} > 0, e_{24} < 0, e_7 > \dfrac{a \left(785 a^2 e_6 + 26 a \sqrt{b} e_8-117 b e_4\right)}{13 b^{3/2}}, e_{35} < 0$ and $e_{92} > \frac{2}{105} \Big(\frac{2 a (9 e_4^3+12 e_4 e_7 e_8-5 e_6 e_7^2)}{\sqrt{b}}-\frac{79 \sqrt{b} e_4 e_7^2}{a}-126 e_4^2 e_7\Big)$ then $V_{9} \cdot V_{11}<0, V_{7} \cdot V_{9}<0, V_{5} \cdot V_{7}<0$ and $V_{3} \cdot V_{5}<0$. In this way we obtain one more limit cycle bifurcating from the origin. Therefore we get at least five limit cycles bifurcating from the origin, concluding the proof.
\end{proof}

%
%Hence we obtain that $|V_3| \ll |V_5| \ll |V_7| \ll |V_{9}| \ll  |V_{11}|$, and if $d_4>0, d_{63}>0, d_{54}<0, d_{15} > 0, d_{26} < 0$ and $d_{37}>0$ then $V_{9} \cdot V_{11}<0, V_{7} \cdot V_{9}<0, V_{5} \cdot V_{7}<0$ and $V_{3} \cdot V_{5}<0$. In this way we obtain one more limit cycle bifurcating from the origin. Therefore we get at least five limit cycles bifurcating from the origin, concluding the proof.

% therefore
%\[
%V_1=1+e_{35}\pi\varepsilon^6+\mathit{O}(\varepsilon^7).
%\]
%So we obtain that $|V_3| \ll |V_5| \ll |V_7| \ll |V_{9}| $ $\ll  |V_{11}|$, and if
%$e_{51} > 0, e_{13} > 0, e_{24} < 0, e_7 > \dfrac{a \left(785 a^2 e_6 + 26 a \sqrt{b} e_8-117 b e_4\right)}{13 b^{3/2}}, e_{35} < 0$ and $e_{92} > \frac{2}{105} \Big(\frac{2 a (9 e_4^3+12 e_4 e_7 e_8-5 e_6 e_7^2)}{\sqrt{b}}-\frac{79 \sqrt{b} e_4 e_7^2}{a}-126 e_4^2 e_7\Big)$ then $V_{9} \cdot V_{11}<0, V_{7} \cdot V_{9}<0, V_{5} \cdot V_{7}<0$ and $V_{3} \cdot V_{5}<0$. Moreover, assuming $e_{35}<0$ then $V_1<1$ and the origin is an attractor. In this way we obtain one more limit cycle bifurcating from the origin. Therefore we get at least five limit cycles bifurcating from the origin, concluding the proof.

\begin{lemma}\label{lema-(2,3),aneg}
	There exist a suitable choice of parameters such that system \eqref{FormaCanonicaR-L-aneg} has 5 limit cycles bifurcating from the homoclinic loop and 2 limit cycles bifurcating from the origin.
\end{lemma}

\begin{proof}
	The proof of the lemma is similar to the proof of Lemma \ref{lema-(2,3),apos}, so we only highlight some minor differences between them. Indeed, we perform replacements of the parameters $e_5$, $e_1$ and $e_2$ in terms of order 1, 2 and 3 in $\varepsilon$, analogously to what we have done for the parameters $d_5$, $d_1$ and $d_2$ in Lemma \ref{lema-(2,3),apos}. Through the replacement of $e_5$ we change the stability of the homoclinic loop $\mathit{L}_1^\varepsilon$. Applying the Poincar\'e-Bendixson Theorem in the convenient annular regions, one of them internal to $\mathit{L}_1^\varepsilon$ and other one external to it, we obtain a limit cycle in each annulus. Proceeding in a complete analogous way we obtain two more limit cycles from the replacement of the parameter $e_1$. Therefore, we get four limit cycles. In order to obtain the fifth limit cycle, we destroy the loop of $\mathit{L}_1^\varepsilon$ doing the referred replacement in the parameter $e_2$. In this case whatever is the sign of $M_2$ (see Section \ref{heteroclinica}) we obtain a limit cycle, being {\it internal} if $M_2<0$ and {\it external} to it otherwise. In any case we obtain five limit cycles bifurcating from the homoclinic loop. To obtain the two limit cycle bifurcating from the weak focus we proceed as in Lemma \ref{lema-(5,0),aneg}.
\end{proof}

\subsection{Proof of Theorem \ref{teorema-apos}}
As commented before the proof consists in apply the lemmas from Subsection \ref{SessaoLemas}. We remark that in the proof of the lemmas we obtain limit cycles of opposite stability for the cases $a>0$, $b<0$ and $a>0$, $b<0$, respectively, first assuming $d_3=0$ and $e_3=0$. Then one more limit cycle bifurcates by suitable perturbations of $d_3$ and $e_3$, respectively, preserving the already obtained limit cycles. 

More precisely in case $a>0$ and $b<0$ the configurations $(s,0)$ with $s\in \{0,1, \dots, 5\}$ follows from Lemma \ref{lema-(5,0),apos}. The case $s=2$ and $m=3$ follows from Lemma \ref{lema-(2,3),apos}. The remaining configurations $(s, m), m \in \{1, 2, 3\}$ and $s+m\leq 5$ follows from Lemma \ref{lema-(s,m),apos}. The proof of the case $a>0$ and $b<0$ is then completed since the configurations of limit cycles of systems \eqref{sistema-R-Lgeneralizado} and \eqref{FormaCanonicaR-L-apos} are the same.

In the case $a<0$ and $b>0$, we proceed in a similar way to the previous case but now the double homoclinic loop plays a role. In what concerns the configurations of limit cycle bifurcating from the weak focus, after a translation of $p_1$ to the origin, configurations $(2s, k)$ with $s\in \{0, 1, \dots, 5\}, k\in \{0,1, 2\}$ follows from both Lemma \ref{lema-(5,0),aneg} and by symmetry with respect to the $y-$axis, see Remark \ref{RemarkTeorema}. Now we prove the other configurations of limit cycles. First, the configuration $(2s, 3m), m\in\{1,2\}$ is obtained by Lemma \ref{lema-(2,3),aneg} because every internal limit cycle bifurcating from the homoclinic loop appears pairwise due to the symmetry.

Therefore, each one of the first two steps of the proof of Lemma \ref{lema-(2,3),aneg} generates three limit cycles, being one external and two internal, so we get $3m$ large limit cycles with $m=1,2$. The $2s$ small limit cycles of the configuration $(2s,3m)$ follows also from Lemma \ref{lema-(2,3),aneg} and symmetry. Finally the configuration $(2s,3m+k)$ is obtained from the previous cases and observing that the value of $k$ is determined from the break of the homoclinic loop. That break could generate one external large limit cycle ($k=1$), or two internal large limit cycles by using symmetry ($k=2$), see Lemma \ref{lema-(2,3),aneg}.

Concerning the realization we notice that $2s+3m+k=12$ can be obtained with $s=m=k=2$ which is precisely the situation of Lemma \ref{lema-(2,3),aneg} when we apply the symmetry. The cases $2s+3m+k<12$ are obtained similarly from the other lemmas from the current section. Finally we notice that the configurations of limit cycles after and before the translation of system \eqref{FormaCanonicaR-L-aneg} are clearly preserved. The same results can be obtained for system \eqref{sistema-R-Lgeneralizado} with $a<0$ and $b>0$. So we are done.

\begin{remark}\label{RemarkTeorema}
We finish remarking some aspects of Theorem \ref{teorema-apos} that must be clarified.
\begin{itemize}
\item[(I)] It does not provide any information about configurations having more than three large limit cycles. Indeed to obtain such configurations one should take into account higher orders of Melnikov function, which is not considered in this paper.
\item[(II)] The value $k$ is determined by broking the loop, then an extra limit cycle emerges. If such a loop is broken in such way that a limit cycle appears internally to the loop surrounding one of the equilibrium, so by the symmetry we have a second limit cycle so $k=2$. Otherwise, we get an externally limit cycle so in this case $k=1$.
	\end{itemize}
\end{remark}

\vs

\noindent {\textbf{Acknowledgments.}} The first and third authors are supported by CNPq-Brazil grant 443302/\-2014-6, PROCAD/CAPES grant 88881.0 68462/2014-01 and PRO\-NEX/FAPEG grant 2017/10267000-508. The second author is partially supported by the Ministerio de Ciencia, Innovación y Universidades, Agencia Estadual de Investigación grant MJM 2016-77278-8 (FEDER), the Agència de Gestió d'Ajuts Universitaris i de Recerca grant 2017SGR1617, and the H2020 European Research Council grant MSCA-RISE-2017-777911.

\end{document}